\documentclass[fleqn, pdftex]{article} 

\usepackage{amsmath,amssymb,amsthm,url,tikz}

\title{Disjunction and existence properties in modal arithmetic}
\author{Taishi Kurahashi\thanks{Email: kurahashi@people.kobe-u.ac.jp} 
\thanks{Graduate School of System Informatics, Kobe University, 1-1 Rokkodai, Nada, Kobe 657-8501, Japan.}
\and
Motoki Okuda\thanks{Independent scholar}}
\date{}

\theoremstyle{plain}
\newtheorem{thm}{Theorem}[section]
\newtheorem{lem}[thm]{Lemma}
\newtheorem{prop}[thm]{Proposition}
\newtheorem{cor}[thm]{Corollary}

\theoremstyle{definition}
\newtheorem{definition}[thm]{Definition}

\newtheorem{remark}[thm]{Remark}
\newtheorem{prob}[thm]{Problem}

\newcommand{\HA}{\mathbf{HA}}
\newcommand{\PA}{\mathbf{PA}}

\newcommand{\NS}{\mathbb{N}}

\newcommand{\K}{\mathbf{K}}
\newcommand{\KF}{\mathbf{K4}}
\newcommand{\KT}{\mathbf{KT}}
\newcommand{\SF}{\mathbf{S4}}
\newcommand{\SFi}{\mathbf{S5}}
\newcommand{\Triv}{\mathbf{Triv}}
\newcommand{\GL}{\mathbf{GL}}
\newcommand{\Ver}{\mathbf{Verum}}

\newcommand{\EA}{\mathbf{EA}}

\newcommand{\PR}{\mathrm{Pr}}
\newcommand{\Prf}{\mathrm{Prf}}
\newcommand{\Con}{\mathrm{Con}}

\newcommand{\LB}{\mathcal{L}_A(\Box)}
\newcommand{\gn}[1]{\ulcorner#1\urcorner}

\newcommand{\B}{\mathrm{B}}
\newcommand{\DB}{\Delta(\mathrm{B})}
\newcommand{\SB}{\Sigma(\mathrm{B})}

\newcommand{\DP}{\mathrm{DP}}
\newcommand{\EP}{\mathrm{EP}}
\newcommand{\MDP}{\mathrm{MDP}}
\newcommand{\MEP}{\mathrm{MEP}}
\newcommand{\DC}{\mathrm{DC}}

\begin{document}

\maketitle

\begin{abstract}
We systematically study several versions of the disjunction and the existence properties in modal arithmetic. 
First, we newly introduce three classes $\B$, $\DB$, and $\SB$ of formulas of modal arithmetic, and study basic properties of them. 
Then, we prove several implications between the properties. 
In particular, among other things, we prove that for any consistent recursively enumerable extension $T$ of $\PA(\K)$ with $T \nvdash \Box \bot$, the $\SB$-disjunction property, the $\SB$-existence property, and the $\B$-existence property are pairwise equivalent. 
Moreover, we introduce the notion of the $\SB$-soundness of theories, and prove that for any consistent recursively enumerable extension of $\PA(\KF)$, the modal disjunction property is equivalent to the $\SB$-soundness. 
\end{abstract}

\section{Introduction}

A theory or a logic $T$ is said to have the \textit{disjunction property} ($\DP$) if for any sentences $\varphi$ and $\psi$ in the language of $T$, if $T \vdash \varphi \lor \psi$, then $T \vdash \varphi$ or $T \vdash \psi$. 
This is a property that may be considered to represent the constructivity of intuitionistic logic. 
G\"odel (\cite{God33}) noted that the intuitionistic propositional logic has $\DP$. 
Gentzen (\cite{Gen34}) and Kleene (\cite{Kle45}) proved that the intuitionistic quantified logic and Heyting arithmetic $\HA$ have $\DP$, respectively. 
A property in arithmetic that is related to $\DP$ is the (numerical) existence property. 
We say that a theory $T$ of arithmetic has the \textit{existence property} ($\EP$) if for any formula $\varphi(x)$ that has no free variables except $x$, if $T \vdash \exists x \varphi(x)$, then $T \vdash \varphi(\overline{n})$ for some natural number $n$. 
Here $\overline{n}$ is the numeral for $n$. 
Kleene (\cite{Kle45}) also proved that $\HA$ has $\EP$. 
Moreover, Friedman (\cite{Fri75}) proved that for any recursively enumerable (r.e.) extension $T$ of $\HA$, $T$ has $\DP$ if and only if $T$ has $\EP$. 

\clearpage

A similar situation has been shown to be true for modal arithmetic. 
Modal arithmetic is a framework of arithmetic equipped with the unary modal operator $\Box$. 
Let $\mathcal{L}_A$ and $\LB$ be the languages of arithmetic and modal arithmetic, respectively. 
A prominent $\LB$-theory of modal arithmetic is $\EA$ (epistemic arithmetic) which is obtained by adding $\SF$ into Peano arithmetic $\PA$. 
The theory $\EA$ was independently introduced by Shapiro (\cite{Sha85}) and Reinhardt (\cite{Rei85,Rei86}). 
In this framework, $\Box$ is intended to represent knowability or informal provability, and the language $\mathcal{L}_A(\Box)$ has the expressive power to make analyses about these concepts.
Moreover, it was shown that $\HA$ is faithfully embeddable into $\EA$ via G\"odel's translation (cf.~\cite{Goo84,Sha85,FF86}). 
This result verifies Shapiro's suggestion that $\EA$ is a system about both classical and intuitionistic mathematics. 
From his suggestion, $\EA$ may possess some constructive properties. 
A theory or a logic $T$ is said to have the \textit{modal disjunction property} ($\MDP$) if for any $\LB$-sentences $\varphi$ and $\psi$, if $T \vdash \Box \varphi \lor \Box \psi$, then $T \vdash \varphi$ or $T \vdash \psi$. 
Also, $T$ is said to have the \textit{modal existence property} ($\MEP$) if for any $\LB$-formula $\varphi(x)$ that has no free variables except $x$, if $T \vdash \exists x \Box \varphi(x)$, then $T \vdash \varphi(\overline{n})$ for some natural number $n$. 
Then, Shapiro (\cite{Sha85}) proved that $\EA$ has both $\MDP$ and $\MEP$. 
Moreover, Friedman and Sheard (\cite{FS89}) proved that for any r.e.~$\LB$-theory $T$ extending $\EA$, $T$ has $\MDP$ if and only if $T$ has $\MEP$\footnote{Actually, Friedman and Sheard proved this theorem for a wider class of $\LB$-theories. This will be discussed in Remark \ref{RemFS}. }. 

In the case of classical logic, $\DP$ is related to the completeness of theories. 
Indeed, it is easy to see that a consistent theory $T$ based on classical logic has $\DP$ if and only if $T$ is complete. 
Hence, G\"odel--Rosser's first incompleteness theorem is restated as follows: 
For any consistent r.e.~extension $T$ of $\PA$, $T$ does not have $\DP$. 
In this context, G\"odel--Rosser's first incompleteness theorem can be strengthened. 
For a class $\Gamma$ of formulas, we say that a theory $T$ has the \textit{$\Gamma$-disjunction property} ($\Gamma$-$\DP$) if for any $\Gamma$ sentences $\varphi$ and $\psi$, if $T \vdash \varphi \lor \psi$, then $T \vdash \varphi$ or $T \vdash \psi$. 
Also, $T$ is said to have the \textit{$\Gamma$-existence property} ($\Gamma$-$\EP$) if for any $\Gamma$ formula $\varphi(x)$ that has no free variables except $x$, if $T \vdash \exists x \varphi(x)$, then $T \vdash \varphi(\overline{n})$ for some natural number $n$. 
Then, it is shown that for any consistent r.e.~extension $T$ of $\PA$, $T$ does not have $\Pi_1$-$\DP$ (see \cite{JE76}). 
On the other hand, for extensions of $\PA$, a similar situation to that of $\DP$ and $\EP$ in intuitionistic logic has been shown to hold.
That is, it is known that $\PA$ has both $\Sigma_1$-$\DP$ and $\Sigma_1$-$\EP$. 
Moreover, Guaspari (\cite{Gua79}) proved that $\Sigma_1$-$\DP$, $\Sigma_1$-$\EP$, and the $\Sigma_1$-soundness are pairwise equivalent for any consistent r.e.~extension of $\PA$. 

In the usual proof of the incompleteness theorems, a provability predicate $\PR_T(x)$, that is, a $\Sigma_1$ formula weakly representing the provability relation of a theory $T$ plays an important role. 
Besides the context in which $\Box$ is intended as informal provability, a modal logical study of the notion of formalized provability has been developed by interpreting $\Box$ in terms of $\PR_T(x)$. 
One of the important results of this study is Solovay's arithmetical completeness theorem which states that if $T$ is $\Sigma_1$-sound, then the propositional modal logic $\GL$ is exactly the logic of all $T$-verifiable principles (\cite{Sol76}). 
In this framework, $\MDP$ also makes sense. 
It is known that $\GL$ enjoys $\MDP$. 
Rather than corresponding to some constructive property, this fact corresponds to the fact that if $T$ is $\Sigma_1$-sound, then $T \vdash \PR_{T}(\gn{\varphi}) \lor \PR_{T}(\gn{\psi})$ implies $T \vdash \varphi$ or $T \vdash \psi$. 

Our motivation for the research in the present paper is to provide a unified viewpoint on $\MDP$ and $\Gamma$-$\DP$, which have been discussed in different contexts and frameworks. 
In particular, we would like to unify the arguments on $\Box$ as an informal provability and $\Box$ as a provability predicate. 
For this purpose, instead of fixing a modal logic such as $\SF$ or $\GL$, we discuss the theory $\PA(L)$ obtained by adding an arbitrary normal modal logic $L$ to $\PA$.
In particular, $\KF$ is a common sublogic of $\SF$ and $\GL$, and thus an investigation for extensions of $\PA(\KF)$ would be applicable to both of the two different interpretations of $\Box$. 
For example, we prove that for any r.e.~$\LB$-theory $T$ extending $\PA(\KF)$, $T$ has $\MDP$ if and only if $T$ has $\MEP$. 
This is a strengthening of the above mentioned form of Friedman and Sheard's result. 

We would also like to analyze the possibility of applying existing methods for studying properties such as $\Gamma$-$\DP$ to modal arithmetic. 
In particular, as suggested by Guaspari's result, $\MDP$ and $\MEP$ may be characterized by soundness with respect to some class of $\LB$-formulas. 
For this reason, in the present paper, we introduce three new classes $\B$, $\DB$, and $\SB$ of $\LB$-formulas. 
Then, we prove that for any consistent r.e.~$\LB$-theory $T$ extending $\PA(\KF)$, $T$ has $\MDP$ if and only if $T$ is $\SB$-sound. 
We also provide a systematic analysis of the disjunction and the existence properties in modal arithmetic, including investigations of $\DP$ and $\EP$ concerning these new classes of formulas.

The present paper is organized as follows. 
In \ref{Sec:MA}, we introduce several theories of modal arithmetic, and show that each of them is a conservative extension of $\PA$. 
In \ref{Sec:Classes}, we introduce three new classes $\B$, $\DB$, and $\SB$ of $\LB$-formulas, and show some basic properties of these classes. 
\ref{Sec:B-DB} is devoted to the study of $\B$-$\DP$, $\DB$-$\DP$, and related properties. 
In \ref{Sec:SB}, we study $\SB$-$\DP$ and related properties. 
In particular, we prove that for any r.e.~extension $T$ of the theory $\PA(\K)$, if $T \nvdash \Box \bot$, then $\SB$-$\DP$, $\SB$-$\EP$, and $\B$-$\EP$ are pairwise equivalent. 
From this result, the equivalence of $\MDP$ and $\MEP$ for any consistent r.e.~$\LB$-theory extending $\PA(\KF)$ is obtained. 
In \ref{Sec:Soundness}, as generalizations of the notions of the soundness and the $\Sigma_1$-soundness of $\mathcal{L}_A$-theories, we introduce the notions of the $\LB$-soundness and the $\SB$-soundness of $\LB$-theories. 
We study these notions precisely, and then, we prove that for any consistent r.e.~extension $T$  of $\PA(\KF)$, $T$ has $\MDP$ if and only $T$ is $\SB$-sound. 
This is a modal arithmetical analogue of Guaspari's theorem. 
Figure \ref{Fig1} summarizes our results obtained in \ref{Sec:B-DB}, \ref{Sec:SB}, and \ref{Sec:Soundness}. 
We also show some non-implications between the properties: 
$\Sigma_1$-soundness does not imply $(\B, \Sigma_1)$-$\DP$ (Proposition \ref{nDP}), $\SB$-$\EP$ does not imply $\MDP$ (Proposition \ref{nclosed}), and $\SB$-$\DC$ does not imply $\B$-$\DP$ (Proposition \ref{ex2}). 
Finally, in the last section, we list several unsolved problems.

\begin{figure}[ht]
\centering
\begin{tikzpicture}
\node (MDP) at (0,6) {MDP};
\node (MEP) at (3.5,6) {MEP};
\node (MSS) at (7,6) {$\SB$-sound};

\node (MSDP) at (0,5) {$\SB$-DP};
\node (MSEP) at (3.5,5) {$\SB$-EP};
\node (WMEP) at (7,5) {$\B$-EP};
\node (WMSS1) at (10.5,5.3) {weakly};
\node (WMSS) at (10.5,5) {$\SB$-sound};

\node (MSSDP) at (3.5,4) {$(\SB, \Sigma_1)$-DP};
\node (BDSDP) at (3.5,3) {$(\DB, \Sigma_1)$-DP};
\node (BSDP) at (3.5,-1.4) {$(\B, \Sigma_1)$-DP};
\node (MSDC) at (7,4) {$\SB$-DC};
\node (BDDC) at (7,3) {$\DB$-DC};
\node (BDC) at (7,-1.4) {$\B$-DC};

\node (BDDP) at (0, 4) {$\DB$-DP};
\node (P) at (0,3.1) {$\vdots$}; 
\node (n+1WMDP) at (0,2.2) {$\B^{n+1}$-DP};
\node (nWMDP) at (0,1.4) {$\B^n$-DP};
\node (P2) at (0,0.5) {$\vdots$}; 
\node (WMDP) at (0,-0.4) {$\B$-DP};
\node (SS) at (3.5,-2.4) {$\Sigma_1$-sound};

\node at (4.3,-1.9) {\scriptsize{Prop.~\ref{DPtoS1sound}}};
\node at (5.5,-1.2) {\scriptsize{Cor.~\ref{CorB-DB}}};
\node at (1.7,-0.6) {\scriptsize{Prop.~\ref{DPtoSDP2}}};
\node at (1,1.8) {\scriptsize{Prop.~\ref{BDP}}};
\node at (1,3) {\scriptsize{Prop.~\ref{BDP}}};
\node at (5.5,3.2) {\scriptsize{Cor.~\ref{CorB-DB}}};
\node at (1.7,3.8) {\scriptsize{Prop.~\ref{DPtoSDP2}}};
\node at (5.5,4.2) {\scriptsize{Cor.~\ref{SBDPDC}}};
\node at (1.8,5.2) {\scriptsize{Thm.~\ref{SBDP}}};
\node at (5.4,5.2) {\scriptsize{Thm.~\ref{SBDP}}};
\node at (8.5,5.2) {\scriptsize{Thm.~\ref{weakSBsound}}};
\node at (1,5.6) {\scriptsize{Prop.~\ref{MDPtoSBDP}}};
\node at (1.8,6.2) {\scriptsize{Cor.~\ref{FSrev}}};
\node at (5,6.2) {\scriptsize{Cor.~\ref{SBsound}}};

\draw [<->, double] (MDP)--(MEP);
\draw [<->, double] (MEP)--(MSS);

\draw [->, double] (MDP)--(MSDP);

\draw [<->, double] (MSDP)--(MSEP);
\draw [<->, double] (MSEP)--(WMEP);
\draw [<->, double] (WMEP)--(WMSS);

\draw [->, double] (MSDP)--(MSSDP);
\draw [->, double] (MSDP)--(BDDP);

\draw [<->, double] (MSSDP)--(MSDC);
\draw [->, double] (MSSDP)--(BDSDP);
\draw [<->, double] (BDSDP)--(BDDC);
\draw [->, double] (BDSDP)--(BSDP);
\draw [<->, double] (BSDP)--(BDC);
\draw [->, double] (BSDP)--(SS);

\draw [->, double] (BDDP)--(BDSDP);
\draw [->, double] (BDDP)--(P);
\draw [->, double] (P)--(n+1WMDP);
\draw [->, double] (n+1WMDP)--(nWMDP);
\draw [->, double] (nWMDP)--(P2);
\draw [->, double] (P2)--(WMDP);
\draw [->, double] (WMDP)--(BSDP);
\draw [->, double] (BSDP)--(SS);
\end{tikzpicture}
\caption{Implications for consistent r.e.~extensions $T$ of $\PA(\KF)$ with $T \nvdash \Box \bot$}\label{Fig1}
\end{figure}

\section{Theories of modal arithmetic}\label{Sec:MA}

We work within the framework of modal arithmetic. 
The language $\LB$ of modal arithmetic consists of logical connectives $\bot, \land, \lor, \to, \neg$, quantifiers $\forall, \exists$, elements of the language $\mathcal{L}_A = \{0, S, +, \times, \leq, =\}$ of first-order arithmetic, and modal operator $\Box$. 
The formulas $\varphi \leftrightarrow \psi$ and $\Diamond \varphi$ are abbreviations for $(\varphi \to \psi) \land (\psi \to \varphi)$ and $\neg \Box \neg \varphi$, respectively.  
A set of sentences is called a \textit{theory}. 
In the present paper, we always assume that the inference rules of every $\LB$-theory are modus ponens (MP) $\dfrac{\varphi \to \psi \quad \varphi}{\psi}$, generalization (\textsc{Gen}) $\dfrac{\varphi}{\forall x \varphi}$, and necessitation (\textsc{Nec}) $\dfrac{\varphi}{\Box \varphi}$. 
Since we will study modal arithmetic from a broader perspective than just $\EA$, we also deal with $\LB$-theories obtained by adding normal modal propositional logics other than $\SF$ into $\PA$. 
Let $\PA_\Box$ be the $\LB$-theory obtained by adding the logical axioms of first-order logic for $\LB$-formulas and the induction axioms for $\LB$-formulas into $\PA$. 
Notice that the value of each $\mathcal{L}_A$-term $t(\vec{x})$ can be effectively computed from the input $\vec{x}$, and thus universal instantiation $\forall x \varphi(x) \to \varphi(t)$ (where $t$ is an $\mathcal{L}_A$-term substitutable for $x$ in $\varphi$), which is problematic in modal predicate logic, is not a problem in our framework. 
As in \cite{Rei85,Sha85,FS89,Dos90}, we adopt universal instantiation as an axiom scheme of $\PA_\Box$. 
Of course, this is not the case in general framework (see \cite[Section 7]{Sha85}). 

For each normal modal propositional logic $L$, let $\PA(L)$ denote the $\LB$-theory obtained by adding universal closures of formulas corresponding to modal axioms of $L$ into $\PA_\Box$. 
We deal with the following $\LB$-theories. 

\begin{itemize}
	\item $\PA(\K) = \PA_\Box + \{\forall \vec{x}(\Box(\varphi \to \psi) \to (\Box \varphi \to \Box \psi)) \mid \varphi, \psi$ are $\LB$-formulas$\}$; 
	\item $\PA(\KF) = \PA(\K) + \{\forall \vec{x}(\Box \varphi \to \Box \Box \varphi) \mid \varphi$ is an $\LB$-formula$\}$; 
	\item $\PA(\KT) = \PA(\K) + \{\forall \vec{x}(\Box \varphi \to \varphi) \mid \varphi$ is an $\LB$-formula$\}$; 
	\item $\PA(\SF) = \EA = \PA(\KT) + \{\forall \vec{x}(\Box \varphi \to \Box \Box \varphi) \mid \varphi$ is an $\LB$-formula$\}$; 
	\item $\PA(\SFi) = \PA(\SF) + \{\forall \vec{x}(\Diamond \varphi \to \Box \Diamond \varphi) \mid \varphi$ is an $\LB$-formula$\}$; 
	\item $\PA(\Triv) = \PA(\K) + \{\forall \vec{x}(\Box \varphi \leftrightarrow \varphi) \mid \varphi$ is an $\LB$-formula$\}$; 
	\item $\PA(\GL) = \PA(\KF) + \{\forall \vec{x}(\Box (\Box \varphi \to \varphi) \to \Box \varphi) \mid \varphi$ is an $\LB$-formula$\}$; 
	\item $\PA(\Ver) = \PA(\K) + \{\Box \bot\}$.  
\end{itemize}
Interestingly, Do\v{s}en \cite[Lemma 7]{Dos90} proved that $\PA(\SFi)$ and $\PA(\Triv)$ are deductively equivalent. 

Here we discuss the principle $x = y \to (\varphi(x) \to \varphi(y))$ of identity. 
Our system has this principle only for atomic formulas $\varphi(x)$ as identity axioms as in the case of classical first-order logic. 
On the other hand, this principle for all $\LB$-formulas is not generally valid in our framework because our language has the symbol $\Box$. 
Shapiro (\cite{Sha85}) states that the following proposition holds for $\PA(\SF)$. 

\begin{prop}\label{PAK}\leavevmode
\begin{enumerate}
	\item $\PA(\K) \vdash x = y \to \Box x = y$. 
	\item For any $\LB$-formula $\varphi(x)$, $\PA(\K) \vdash x = y \to (\varphi(x) \to \varphi(y))$. 
\end{enumerate}
\end{prop}
\begin{proof}
1. Let $\varphi(x, y)$ be the formula $x = y \to \Box x = y$. 
Firstly, we prove $\PA_\Box \vdash \forall y \varphi(0, y)$. 
Since $\PA \vdash 0 = 0$, we have $\PA_\Box \vdash \Box 0 = 0$, and hence $\PA_\Box \vdash \varphi(0, 0)$. 
Since $\PA \vdash 0 \neq S(y)$, we also have $\PA_\Box \vdash \varphi(0, S(y))$, and thus $\PA_\Box \vdash \forall y(\varphi(0, y) \to \varphi(0, S(y)))$. 
By the induction axiom for $\varphi(0, y)$, we obtain $\PA_\Box \vdash \forall y \varphi(0, y)$. 

Secondly, we prove $\PA(\K) \vdash \forall y \varphi(x, y) \to \forall y \varphi(S(x), y)$. 
Since $\PA \vdash S(x) \neq 0$, we have $\PA_\Box \vdash \varphi(S(x), 0)$. 
It follows from $\PA \vdash S(x) = S(y) \to x = y$ that $\PA_\Box \vdash \varphi(x, y) \land S(x) = S(y) \to \Box x = y$. 
Since $\PA \vdash x = y \to S(x) = S(y)$, we have $\PA(\K) \vdash \Box x = y \to \Box S(x) = S(y)$. 
Thus, we get 
\[
	\PA(\K) \vdash \varphi(x, y) \land S(x) = S(y) \to \Box S(x) = S(y).
\] 
This means $\PA(\K) \vdash \varphi(x, y) \to \varphi(S(x), S(y))$. 
By the universal instantiation, we have $\PA(\K) \vdash \forall y \varphi(x, y) \to \varphi(S(x), S(y))$, and hence 
\[
	\PA(\K) \vdash \forall y\varphi(x, y) \to \forall y(\varphi(S(x), y) \to \varphi(S(x), S(y))).
\] 
From this with $\PA_\Box \vdash \varphi(S(x), 0)$, we obtain $\PA(\K) \vdash \forall y \varphi(x, y) \to \forall y \varphi(S(x), y)$ by the induction axiom for $\varphi(S(x), y)$. 

Finally, by the induction axiom for $\forall y \varphi(x, y)$, we conclude $\PA(\K) \vdash \forall x \forall y \varphi(x, y)$. 

2. This is proved by induction on the construction of $\varphi(x)$. 
We only prove the case that $\varphi(x)$ is of the form $\Box \psi(x)$ and the statement holds for $\psi(x)$. 
By the induction hypothesis, $\PA(\K) \vdash x = y \to (\psi(x) \to \psi(y))$. 
Then, $\PA(\K)$ proves $\Box x = y \to (\Box \psi(x) \to \Box \psi(y))$. 
By combining this with Clause 1, we conclude $\PA(\K) \vdash x = y \to (\Box \psi(x) \to \Box \psi(y))$.
\end{proof}

We say that a theory $T$ is a \textit{subtheory} of a theory $U$, $U \vdash T$, if every axiom of $T$ is provable in $U$. 
Makinson's theorem (\cite{Mak71}) states that every consistent normal modal propositional logic $L$ is a sublogic of $\Triv$ or $\Ver$ (see also \cite{HC96}). 
Hence, every $\LB$-theory of the form $\PA(L)$ for some consistent normal propositional modal logic $L$ is a subtheory of $\PA(\Triv)$ or $\PA(\Ver)$. 
We prove that every such logic is a conservative extension of $\PA$. 

First, we prove that $\PA(\Triv)$ is a conservative extension of $\PA$. 
In order to prove this, we introduce a translation $\alpha$ of $\LB$-formulas into $\mathcal{L}_A$-formulas. 

\begin{definition}[$\alpha$-translation]
We define a translation $\alpha$ of $\LB$-formulas into $\mathcal{L}_A$-formulas inductively as follows: 
\begin{enumerate}
	\item If $\varphi$ is an $\mathcal{L}_A$-formula, then $\alpha(\varphi) : \equiv \varphi$; 
	\item $\alpha$ preserves logical connectives and quantifiers; 
	\item $\alpha(\Box \varphi) : \equiv \alpha(\varphi)$. 
\end{enumerate}
\end{definition}

It is obvious that for any $\LB$-formula $\varphi$, $\PA(\Triv) \vdash \varphi \leftrightarrow \alpha(\varphi)$. 
Moreover, 

\begin{prop}\label{alpha}
For any $\LB$-formula $\varphi$, if $\PA(\Triv) \vdash \varphi$, then $\PA \vdash \alpha(\varphi)$. 
\end{prop}
\begin{proof}
We prove the proposition by induction on the length of proofs of $\varphi$ in $\PA(\Triv)$. 
\begin{itemize}
	\item If $\varphi$ is an axiom of $\PA$, then $\alpha(\varphi) \equiv \varphi$ and $\PA \vdash \alpha(\varphi)$.
	\item If $\varphi$ is a logical axiom, then so is $\alpha(\varphi)$, and it is $\PA$-provable. 
	\item If $\varphi$ is an induction axiom in the language $\LB$, then $\alpha(\varphi)$ is also an induction axiom in $\mathcal{L}_A$, and so $\PA \vdash \alpha(\varphi)$. 
	\item If $\varphi$ is $\forall \vec{x}(\Box(\psi \to \sigma) \to (\Box \psi \to \Box \sigma))$, then $\alpha(\varphi)$ is the $\PA$-provable sentence $\forall \vec{x} \bigl((\alpha(\psi) \to \alpha(\sigma)) \to (\alpha(\psi) \to \alpha(\sigma)) \bigr)$. 
	\item If $\varphi$ is $\forall \vec{x} (\psi \leftrightarrow \Box \psi)$, then $\alpha(\varphi)$ is $\forall \vec{x}(\alpha(\psi) \leftrightarrow \alpha(\psi))$. 
	This is provable in $\PA$. 
	\item If $\varphi$ is derived from $\psi$ and $\psi \to \varphi$ by MP, then by the induction hypothesis, $\PA \vdash \alpha(\psi)$ and $\PA \vdash \alpha(\psi) \to \alpha(\varphi)$, and hence $\PA \vdash \alpha(\varphi)$. 
	\item If $\varphi$ is derived from $\psi(x)$ by \textsc{Gen}, then $\varphi \equiv \forall x \psi(x)$. 
	By the induction hypothesis, $\PA \vdash \alpha(\psi(x))$ and hence $\PA \vdash \forall x \alpha(\psi(x))$. 
	Therefore, $\PA \vdash \alpha(\varphi)$. 
	\item If $\varphi$ is derived from $\psi$ by \textsc{Nec}, then $\varphi \equiv \Box \psi$. 
	By the induction hypothesis, $\PA \vdash \alpha(\psi)$. 
	Since $\alpha(\varphi) \equiv \alpha(\psi)$, we have $\PA \vdash \alpha(\varphi)$. 
\end{itemize}
\end{proof}

Let $\NS$ be the standard model of arithmetic in the language $\mathcal{L}_A$. 
We say that an $\LB$-theory $T$ is \textit{$\mathcal{L}_A$-sound} if for any $\mathcal{L}_A$-sentence $\varphi$, $\NS \models \varphi$ whenever $T \vdash \varphi$.

\begin{cor}\label{Triv}
$\PA(\Triv)$ is a conservative extension of $\PA$. 
In particular, $\PA(\Triv)$ is $\mathcal{L}_A$-sound. 
\end{cor}
\begin{proof}
Let $\varphi$ be any $\mathcal{L}_A$-sentence such that $\PA(\Triv) \vdash \varphi$. 
By Proposition \ref{alpha}, $\PA \vdash \alpha(\varphi)$. 
Since $\alpha(\varphi) \equiv \varphi$, $\PA \vdash \varphi$. 
Furthermore, by the $\mathcal{L}_A$-soundness of $\PA$, $\PA(\Triv)$ is also $\mathcal{L}_A$-sound. 
\end{proof}

Next, we prove that $\PA(\Ver)$ is a conservative extension of $\PA$. 
We also introduce another translation $\beta$.

\begin{definition}[$\beta$-translation]
We define a translation $\beta$ of $\LB$-formulas into $\mathcal{L}_A$-formulas inductively as follows: 
\begin{enumerate}
	\item If $\varphi$ is an $\mathcal{L}_A$-formula, then $\beta(\varphi) : \equiv \varphi$; 
	\item $\beta$ preserves logical connectives and quantifiers; 
	\item $\beta(\Box \varphi) : \equiv 0=0$. 
\end{enumerate}
\end{definition}

As in the case of $\alpha$, for any $\LB$-formula $\varphi$, $\PA(\Ver) \vdash \varphi \leftrightarrow \beta(\varphi)$. 
Moreover, 

\begin{prop}\label{beta}
For any $\LB$-formula $\varphi$, if $\PA(\Ver) \vdash \varphi$, then $\PA \vdash \beta(\varphi)$. 
\end{prop}
\begin{proof}
As in the proof of Proposition \ref{alpha}, this proposition is proved by induction on the length of proofs of $\varphi$ in $\PA(\Ver)$. 
We only give proofs of the following three cases: 
\begin{itemize}
	\item If $\varphi$ is $\forall \vec{x}(\Box(\psi \to \sigma) \to (\Box \psi \to \Box \sigma))$, then $\beta(\varphi)$ is the $\PA$-provable sentence $\forall \vec{x} (0=0 \to (0=0 \to 0=0))$. 
	\item If $\varphi$ is $\Box \bot$, then $\beta(\varphi)$ is the $\PA$-provable sentence $0=0$. 
	\item If $\varphi$ is derived from $\psi$ by \textsc{Nec}, then $\varphi \equiv \Box \psi$. 
	Since $\beta(\varphi) \equiv 0=0$, this is $\PA$-provable.  
\end{itemize}
\end{proof}

\begin{cor}\label{Ver}
$\PA(\Ver)$ is a conservative extension of $\PA$. 
In particular, $\PA(\Ver)$ is $\mathcal{L}_A$-sound. 
\end{cor}

We close this section by showing that the notion of $\Sigma_1$ formulas has a high affinity with modal arithmetic. 
The following theorem is proved by applying a schematic proof of formalized $\Sigma_1$-completeness theorem (see \cite{Buc93, Kur20,Rau10}). 
This is also implicitly stated in Friedman and Sheard (\cite{FS89}). 

\begin{thm}[Formalized $\Sigma_1$-completeness theorem]\label{S1compl}
For any $\Sigma_1$ formula $\varphi$, we have $\PA(\K) \vdash \varphi \to \Box \varphi$. 
\end{thm}
\begin{proof}
Before proving the theorem, we show that for any $\Sigma_1$ formula $\psi$, there exists a $\Sigma_1$ formula $\psi'$ such that $\PA \vdash \psi \leftrightarrow \psi'$, $\psi'$ does not contain the connectives $\neg$ and $\to$, and every atomic formula contained in $\psi'$ is of the form $t_1 = t_2$ for some $\mathcal{L}_A$-terms $t_1$ and $t_2$. 
First, we easily find a $\Sigma_1$ formula $\psi_0$ without the connective $\to$ such that $\psi_0$ is logically equivalent to $\psi$ and every negation symbol $\neg$ in $\psi_0$ is applied to an atomic formula. 
Then, by replacing every negated atomic formula $\neg (t_1 = t_2)$ or $\neg (t_1 < t_2)$ of $\psi_1$ by $t_1 < t_2 \lor t_2 < t_1$ or $t_1 = t_2 \lor t_2 < t_1$ respectively, we obtain a $\PA$-equivalent $\Sigma_1$ formula $\psi_1$ without having $\neg$. 
Finally, by replacing every atomic formula $t_1 < t_2$ of $\psi_1$ by $\exists y(t_1 + S(y) = t_2)$, we obtain a required equivalent $\Sigma_1$ formula $\psi'$. 
Since $\PA(\K) \vdash \Box \psi \leftrightarrow \Box \psi'$, to prove the theorem, it suffices to show that $\PA(\K) \vdash \varphi \to \Box \varphi$ for any $\Sigma_1$ formula $\varphi$ such that it does not contain the connectives $\neg$ and $\to$, and that every atomic formula contained in $\varphi$ is of the form $t_1 = t_2$ for some $\mathcal{L}_A$-terms $t_1$ and $t_2$.

This is proved by induction on the construction of $\varphi$. 
If $\varphi$ is $t_1 = t_2$, then $\PA(\K) \vdash t_1 = t_2 \to \Box t_1 = t_2$ follows from $\PA(\K) \vdash x = y \to \Box x = y$ (Proposition \ref{PAK}) by substituting $t_1$ and $t_2$ into $x$ and $y$, respectively. 
The cases for $\land$, $\lor$, $\forall x < t$ and $\exists$ are proved as in the proof of Theorem \ref{SBcompl} below. 
\end{proof}

\section{Classes of $\LB$-formulas}\label{Sec:Classes}

In first-order arithmetic, it is important to classify $\mathcal{L}_A$-formulas according to the arithmetic hierarchy. 
In this section, we introduce three classes $\B$, $\DB$, and $\SB$ of $\LB$-formulas, and investigate basic properties of formulas in these classes. 
Our classes $\DB$ and $\SB$ are modal arithmetical counterparts of $\Delta_0$ and $\Sigma_1$, respectively.

\begin{definition}[$\B$, $\DB$ and $\SB$]\leavevmode
\begin{itemize}
	\item Let $\B$ be the class of all $\LB$-formulas of the form $\Box \varphi$. 
	\item Let $\DB$ be the smallest class of $\LB$-formulas satisfying the following conditions: 
\begin{enumerate}
	\item $\Delta_0 \cup \B \subseteq \DB$; 
	\item If $\varphi$ and $\psi$ are in $\DB$, then so are $\varphi \land \psi$, $\varphi \lor \psi$, $\forall x < t\, \varphi$ and $\exists x < t\, \varphi$, where $t$ is an $\mathcal{L}_A$-term in which $x$ does not occur. 
\end{enumerate}
	\item Let $\SB$ be the smallest class of $\LB$-formulas satisfying the following conditions: 
\begin{enumerate}
	\item $\Sigma_1 \cup \B \subseteq \SB$;  
	\item If $\varphi$ and $\psi$ are in $\SB$, then so are $\varphi \land \psi$, $\varphi \lor \psi$, $\exists x \varphi$ and $\forall x < t\, \varphi$, where $t$ is an $\mathcal{L}_A$-term in which $x$ does not occur. 
\end{enumerate}
\end{itemize}
\end{definition}

We emphasize here that in some sense the class $\SB$ is a natural extension of the class $\Sigma_1$. 
For each r.e.~theory $T$, let $\PR_T(x)$ be a fixed $\Sigma_1$ provability predicate of $T$. 
In the context of interpreting $\Box$ by $\PR_T(x)$, each $\LB$-formula of the form $\Box \varphi$ is interpreted by a $\Sigma_1$ formula, and hence every $\SB$ formula is also recognized as a $\Sigma_1$ formula. 
From this perspective, we will attempt to extend the properties possessed by $\Sigma_1$ formulas in first-order arithmetic to $\SB$ formulas in modal arithmetic.
Note, however, that $\DB$, unlike $\Delta_0$, is not closed under taking negation and implication.
For example, it can be shown that there is no $\DB$ sentence $\varphi$ such that $\PA(\K) \vdash \neg \Box \bot \leftrightarrow \varphi$ (see Corollary \ref{NonDB} below). 

The following proposition states that the relationship between $\SB$ and $\DB$ is similar to the relationship between $\Sigma_1$ and $\Delta_0$. 

\begin{prop}\label{MSandBD}
For any $\SB$ formula $\varphi$, there exist a variable $v$ and a $\DB$ formula $\psi$ such that $\PA_\Box \vdash \varphi \leftrightarrow \exists v \psi$. 
\end{prop}
\begin{proof}
We prove the proposition by induction on the construction of $\varphi$. 

\begin{itemize}
	\item If $\varphi$ is $\Sigma_1$, then there exists a $\Delta_0$ formula $\psi$ such that $\PA \vdash \varphi \leftrightarrow \exists v \psi$. 

	\item If $\varphi$ is of the form $\Box \varphi_0$, then $\varphi \in \DB$ and $\PA_\Box \vdash \varphi \leftrightarrow \exists v \varphi$ for a variable $v$ not contained in $\varphi$. 

	\item Let $\circ \in \{\land, \lor\}$. 
	If $\varphi$ is of the form $\varphi_0 \circ \varphi_1$, then by the induction hypothesis, there exist distinct variables $v_0$ and $v_1$ and $\DB$ formulas $\psi_0$ and $\psi_1$ such that $\PA_\Box$ proves $\varphi_0 \leftrightarrow \exists v_0 \psi_0$ and $\varphi_1 \leftrightarrow \exists v_1 \psi_1$. 
Let $v$ be any variable that does not occur in $\psi_0$ or $\psi_1$, and is not $v_0$ or $v_1$. 
Then, $\PA_\Box$ proves the equivalence $\varphi \leftrightarrow \exists v \, \exists v_0 < v \, \exists v_1 < v \, (\psi_0 \circ \psi_1)$. 

	\item If $\varphi$ is of the form $\exists x \varphi_0$, then by the induction hypothesis, there exist a variable $v_0$ and a $\DB$ formula $\psi_0$ such that $\PA \vdash \varphi_0 \leftrightarrow \exists v_0 \psi_0$. 
	Let $v$ be any variable not contained in $\psi_0$ and is not $v_0$ or $x$. 
	Then, $\PA_\Box$ proves the equivalence $\varphi \leftrightarrow \exists v \, \exists x < v \, \exists v_0 < v \, \psi_0$. 

	\item The case that $\varphi$ is of the form $\exists x < t\, \varphi_0$, where $t$ is an $\mathcal{L}_A$-term in which $x$ does not occur is proved as in the proof of the case of $\exists$. 

	\item Suppose $\varphi$ is of the form $\forall x < t\, \varphi_0$, where $t$ is an $\mathcal{L}_A$-term in which $x$ does not occur. 
	By the induction hypothesis, there exists a variable $v_0$ and a $\DB$ formula $\psi_0$ such that $\PA_\Box \vdash \varphi_0 \leftrightarrow \exists v_0 \psi_0$. 
	Then, $\PA_\Box \vdash \varphi \leftrightarrow \forall x < t\, \exists v_0 \psi_0$. 
	By the collection principle for $\LB$-formulas derived from the induction axioms for $\LB$-formulas, we obtain
\[
	 \PA_\Box \vdash \varphi \leftrightarrow \exists v\, \forall x < t\, \exists v_0 < v\, \psi_0
\]
for some appropriate variable $v$. 
\end{itemize}
\end{proof}

\begin{prop}\label{DBSent}
For any $\DB$ sentence $\varphi$, there exist a natural number $k$ and sentences $\psi_0, \ldots, \psi_{k-1}$ such that $\PA_\Box \vdash \varphi \leftrightarrow \bigvee_{i < k} \Box \psi_i$. 
Here $\bigvee_{i < 0} \Box \psi_i$ denotes $\bot$.
\end{prop}
\begin{proof}
We prove the proposition by induction on the construction of $\varphi$. 

\begin{itemize}
	\item If $\varphi$ is a $\Delta_0$ sentence, then either $\PA \vdash \varphi$ or $\PA \vdash \neg \varphi$. 
	Thus, $\PA_\Box \vdash \varphi \leftrightarrow \Box 0=0$ or $\PA_\Box \vdash \varphi \leftrightarrow \bot$. 

	\item If $\varphi$ is of the form $\Box \psi$, then the statement is trivial. 

	\item 	If $\varphi$ is of the form $\psi \land \sigma$, then there exist sentences $\xi_0, \ldots, \xi_{k-1}, \eta_0, \ldots, \eta_{l-1}$ such that $\PA_\Box \vdash \psi \leftrightarrow \bigvee_{i < k} \Box \xi_i$ and  $\PA_\Box \vdash \sigma \leftrightarrow \bigvee_{j < l} \Box \eta_j$. 
	Then, $\PA_\Box \vdash \varphi \leftrightarrow \bigvee_{i < k} \bigvee_{j < l} \Box(\xi_i \land \eta_j)$. 
	
	\item 	If $\varphi$ is of the form $\psi \lor \sigma$, then the statement is obvious by the induction hypothesis. 
	
	\item If $\varphi$ is of the form $\exists y < t\, \psi(y)$ for some $\mathcal{L}_A$-term $t$, then $t$ is a closed term because $\varphi$ is a sentence. 
	Let $m$ be the value of $t$, then $\PA_\Box \vdash \varphi \leftrightarrow \bigvee_{i < m} \psi(\overline{i})$. 
	Then, the statement holds by the induction hypothesis. 
	
	\item If $\varphi$ is of the form $\forall y < t\, \psi(y)$ for some closed $\mathcal{L}_A$-term $t$, then for the value $m$ of the term $t$, $\PA_\Box \vdash \varphi \leftrightarrow \bigwedge_{i < m} \psi(\overline{i})$. 
	We can prove the statement by the induction hypothesis as in the proof of the case $\land$. 
\end{itemize}
\end{proof}

\begin{cor}\label{NonDB}
Let $T$ be a theory extending $\PA(\K)$ such that $T \nvdash \Box \bot$ and $T \nvdash \neg \Box \bot$. 
Then, there is no $\DB$ sentence $\varphi$ such that $T \vdash \neg \Box \bot \leftrightarrow \varphi$.  
\end{cor}
\begin{proof}
Suppose, towards a contradiction, that $\varphi$ is a $\DB$ sentence such that $T \vdash \neg \Box \bot \leftrightarrow \varphi$. 
By Proposition \ref{DBSent}, there exist $k$ and $\psi_0, \ldots, \psi_{k-1}$ such that $\PA_{\Box} \vdash \varphi \leftrightarrow \bigvee_{i < k} \Box \psi_i$. 
Then, $T \vdash \neg \Box \bot \leftrightarrow \bigvee_{i < k} \Box \psi_i$. 
Since $T \nvdash \Box \bot$, we get $k > 0$. 
Then, $T \vdash \Box \psi_0 \to \neg \Box \bot$. 
On the other hand, $T \vdash \neg \Box \psi_0 \to \neg \Box \bot$ because $T$ is an extension of $\PA(\K)$. 
Therefore, we obtain $T \vdash \neg \Box \bot$. 
This is a contradiction. 
\end{proof}

We naturally extend Theorem \ref{S1compl} into the framework of modal arithmetic.

\begin{thm}[Formalized $\SB$-completeness theorem]\label{SBcompl}
For any $\varphi \in \SB$, $\PA(\KF) \vdash \varphi \to \Box \varphi$. 
\end{thm}
\begin{proof}
We prove the theorem by induction on the construction of $\varphi$. 
\begin{itemize}
	\item If $\varphi$ is a $\Sigma_1$ formula, then $\PA(\K) \vdash \varphi \to \Box \varphi$ by Theorem \ref{S1compl}. 

	\item If $\varphi$ is of the form $\Box \psi$, then $\PA(\KF) \vdash \varphi \to \Box \varphi$. 

	\item If $\varphi$ is $\psi \land \sigma$, then by the induction hypothesis, $\PA(\KF) \vdash \varphi \to \Box \psi \land \Box \sigma$. 
We have $\PA(\KF) \vdash \varphi \to \Box \varphi$. 

	\item If $\varphi$ is $\psi \lor \sigma$, then by the induction hypothesis, $\PA(\KF)$ proves $\psi \to \Box (\psi \lor \sigma)$ and $\sigma \to \Box (\psi \lor \sigma)$. 
Hence, $\PA(\KF) \vdash \varphi \to \Box \varphi$. 

	\item Suppose that $\varphi$ is $\exists x \psi$. 
Since $\PA(\K) \vdash \psi \to \exists x \psi$, we have $\PA(\K) \vdash \Box \psi \to \Box \exists x \psi$. 
By the induction hypothesis, $\PA(\K) \vdash \psi \to \Box \psi$. 
Thus, $\PA(\K) \vdash \psi \to \Box \exists x \psi$. 
Then, we obtain $\PA(\KF) \vdash \varphi \to \Box \varphi$. 

	\item Before proving the case that $\varphi$ is of the form $\forall x < t\, \psi(x)$ generally, we prove the restricted case that $t$ is some variable $y$ not occurring in $\psi$. 
	Suppose that $\varphi(y)$ is of the form $\forall x < y\, \psi(x)$ for some variable $y$ not occurring in $\psi(x)$. 
	Let $\xi(y)$ be the formula $\varphi(y) \to \Box \varphi(y)$, and then we prove $\PA(\K) \vdash \forall y \xi(y)$ by using the induction axiom. 

	For the base step, since trivially $\PA(\K) \vdash \forall x < 0\, \psi(x)$, we have $\PA(\K) \vdash \Box \varphi(0)$ and hence $\PA(\K) \vdash \xi(0)$. 

	For the induction step, since 
\[
	\PA(\K) \vdash [\forall x < S(y)\, \psi(x)] \leftrightarrow [(\forall x < y\, \psi(x)) \land \psi(y)], 
\]
we have 
\begin{equation}\label{eq2}
	\PA(\K) \vdash \varphi(S(y)) \leftrightarrow [\varphi(y) \land \psi(y)].
\end{equation} 
By the induction hypothesis, $\PA(\KF) \vdash \psi(y) \to \Box \psi(y)$. 
By combining this with (\ref{eq2}) and the definition of $\xi(y)$, 
\[
	\PA(\KF) \vdash \xi(y) \land \varphi(S(y)) \to \Box \varphi(y) \land \Box \psi(y). 
\]
Then, by (\ref{eq2}) again, 
\[
	\PA(\KF) \vdash \xi(y) \land \varphi(S(y)) \to \Box \varphi(S(y)).
\] 
Equivalently, $\PA(\KF) \vdash \xi(y) \to \xi(S(y))$. 

Therefore, by the induction axiom, we conclude $\PA(\KF) \vdash \forall y \xi(y)$. 

Finally, suppose that $\varphi$ is of the form $\forall x < t\, \psi$ for some $\mathcal{L}_A$-term $t$. 
We have already proved that $\PA(\KF)$ proves $\forall x < y\, \psi \to \Box \forall x < y\, \psi$ for some variable $y$ not occurring in $\psi$. 
By substituting $t$ for $y$ in this formula, we obtain $\PA(\KF) \vdash \varphi \to \Box \varphi$. 
\end{itemize}
\end{proof}

\begin{cor}[$\SB$-deduction theorem]\label{SBDT}
Let $T$ be any extension of $\PA(\KF)$ and let $X$ be any set of $\SB$ sentences. 
Then, for any $\LB$-formula $\varphi$, if $T + X \vdash \varphi$, then there exist $\sigma_0, \ldots, \sigma_{k-1} \in X$ such that $T \vdash \sigma_0 \land \cdots \land \sigma_{k-1} \to \varphi$. 
\end{cor}
\begin{proof}
This is proved by induction on the length of a proof of $\varphi$ in $T + X$. 
We only give a proof of the case that $\varphi$ is derived from $\psi$ by the rule \textsc{Nec}. 
Then, $\varphi$ is of the form $\Box \psi$. 
By the induction hypothesis, $T \vdash \sigma_0 \land \cdots \land \sigma_{k-1} \to \psi$ for some $\sigma_0, \ldots, \sigma_{k-1} \in X$. 
Then, $T \vdash \Box \sigma_0 \land \cdots \land \Box \sigma_{k-1} \to \Box \psi$. 
By Theorem \ref{SBcompl}, $T \vdash \sigma_0 \land \cdots \land \sigma_{k-1} \to \Box \psi$. 
\end{proof}

\section{$\B$-$\DP$, $\DB$-$\DP$ and related properties}\label{Sec:B-DB}

We introduce several versions of the partial disjunction property. 

\begin{definition}
Let $T$ be a theory and let $\Gamma$ and $\Theta$ be classes of formulas. 
\begin{itemize}
	\item $T$ is said to have the \textit{modal disjunction property} \textup{(}$\MDP$\textup{)} if for any sentences $\varphi$ and $\psi$, if $T \vdash \Box \varphi \lor \Box \psi$, then $T \vdash \varphi$ or $T \vdash \psi$. 

	\item $T$ is said to have the \textit{modal existence property} \textup{(}$\MEP$\textup{)} if for any formula $\varphi(x)$ that has no free variables except $x$, if $T \vdash \exists x \Box \varphi(x)$, then for some natural number $n$, $T \vdash \varphi(\overline{n})$. 

	\item $T$ is said to have the \textit{$\Gamma$-disjunction property} \textup{(}$\Gamma$-$\DP$\textup{)} if for any $\Gamma$ sentences $\varphi$ and $\psi$, if $T \vdash \varphi \lor \psi$, then $T \vdash \varphi$ or $T \vdash \psi$. 

	\item $T$ is said to have the \textit{$\Gamma$-existence property} \textup{(}$\Gamma$-$\EP$\textup{)} if for any $\Gamma$ formula $\varphi(x)$ that has no free variables except $x$, if $T \vdash \exists x \varphi(x)$, then for some natural number $n$, $T \vdash \varphi(\overline{n})$. 

	\item $T$ is said to have the \textit{$(\Gamma, \Theta)$-disjunction property} \textup{(}$(\Gamma, \Theta)$-$\DP$\textup{)} if for any $\Gamma$ sentence $\varphi$ and any $\Theta$ sentence $\psi$, if $T \vdash \varphi \lor \psi$, then $T \vdash \varphi$ or $T \vdash \psi$. 
	
	\item For $n \geq 2$, $T$ is said to have the \textit{$n$-fold $\B$-disjunction property} \textup{(}$\B^n$-$\DP$\textup{)} if for any $\LB$-sentences $\varphi_1, \ldots, \varphi_n$, if $T \vdash \Box \varphi_1 \lor \cdots \lor \Box \varphi_n$, then $T \vdash \Box \varphi_i$ for some $i$ $(1 \leq i \leq n)$. 

	\item If $T$ is r.e., then $T$ is said to be \textit{$\Gamma$-disjunctively correct} \textup{(}$\Gamma$-$\DC$\textup{)} if for any $\Gamma$ sentence $\varphi$, if $T \vdash \varphi \lor \PR_T(\gn{\varphi})$, then $T \vdash \varphi$. 

	\item We say that $T$ is closed under the \textit{box elimination rule} if for any sentence $\varphi$, if $T \vdash \Box \varphi$, then $T \vdash \varphi$. 
\end{itemize}
\end{definition}

Here $\PR_T(x)$ is a fixed natural $\Sigma_1$ provability predicate of $T$. 
We also fix a primitive recursive proof predicate $\Prf_T(x, y)$ of $T$ saying that $y$ encodes a $T$-proof of $x$, whose existence is guaranteed by Craig's trick. 

Of course, $(\Gamma, \Gamma)$-$\DP$ and $\B^2$-$\DP$ are exactly $\Gamma$-$\DP$ and $\B$-$\DP$, respectively. 
The notion of $\Gamma$-$\DC$ was introduced in \cite{Kur18}. 
It is known that for any consistent r.e.~extension $T$ of $\PA$, $T$ is $\Sigma_1$-$\DC$ if and only if $T$ is $\Sigma_1$-sound (cf.~\cite{Kur18}). 

\begin{prop}\label{BDP}
Let $T$ be any extension of $\PA_\Box$. 
\begin{enumerate}
	\item For any $n \geq 2$, if $T$ has $\B^{n+1}$-$\DP$, then $T$ also has $\B^n$-$\DP$; 
	\item $T$ has $\B^n$-$\DP$ for all $n \geq 2$ if and only if $T$ has $\DB$-$\DP$.  
\end{enumerate}
\end{prop}
\begin{proof}
1. Let $\varphi_1, \ldots, \varphi_n$ be any sentences such that $T \vdash \Box \varphi_1 \lor \cdots \lor \Box \varphi_n$. 
Then, $T \vdash \Box \varphi_1 \lor \cdots \lor \Box \varphi_n \lor \Box \varphi_n$. 
By $\B^{n+1}$-$\DP$, for some $i$ ($1 \leq i \leq n$), we have $T \vdash \Box \varphi_i$. 

2. $(\Rightarrow)$: Let $\varphi$ and $\psi$ be any $\DB$ sentences such that $T \vdash \varphi \lor \psi$. 
By Proposition \ref{DBSent}, there exist sentences $\varphi_0, \ldots, \varphi_{k-1}$ and $\psi_0, \ldots, \psi_{l-1}$ such that $T \vdash \varphi \leftrightarrow \bigvee_{i < k} \Box \varphi_i$ and $T \vdash \psi \leftrightarrow \bigvee_{j < l} \Box \psi_j$. 
Then, $T \vdash \bigvee_{i < k} \Box \varphi_i \lor \bigvee_{j < l} \Box \psi_j$. 
If $k = 0$ or $l = 0$, then we easily obtain $T \vdash \varphi$ or $T \vdash \psi$. 
Thus, we may assume both $k$ and $l$ are larger than $0$. 
Then, $k+l \geq 2$. 
By $\B^{k+l}$-$\DP$, there exists $i < k$ or $j < l$ such that $T \vdash \Box \varphi_i$ or $T \vdash \Box \psi_j$. 
Then, we obtain that $T \vdash \varphi$ or $T \vdash \psi$. 

$(\Leftarrow)$: We prove this implication by induction on $n \geq 2$. 
Since $\B \subseteq \DB$, $T$ has $\B^2$-$\DP$. 
Suppose that $T$ has $\B^n$-$\DP$ and we would like to prove that $T$ also has $\B^{n+1}$-$\DP$. 
Let $\varphi_1, \ldots, \varphi_n, \varphi_{n+1}$ be any sentences such that $T \vdash \Box \varphi_1 \lor \cdots \lor \Box \varphi_n \lor \Box \varphi_{n+1}$. 
Since both $\Box \varphi_1 \lor \cdots \lor \Box \varphi_n$ and $\Box \varphi_{n+1}$ are $\DB$ sentences, we have $T \vdash \Box \varphi_1 \lor \cdots \lor \Box \varphi_n$ or $T \vdash \Box \varphi_{n+1}$ by $\DB$-$\DP$. 
In the former case, we obtain $T \vdash \Box \varphi_i$ for some $i$ ($1 \leq i \leq n$) by the induction hypothesis. 
We have proved that $T$ has $\B^{n+1}$-$\DP$. 
\end{proof}

The following proposition is immediate from the definitions. 

\begin{prop}\label{DPEPb}
Let $T$ be any $\LB$-theory. 
\begin{enumerate}
	\item $T$ has $\MDP$ if and only if $T$ has $\B$-$\DP$ and is closed under the box elimination rule; 
	\item $T$ has $\MEP$ if and only if $T$ has $\B$-$\EP$ and is closed under the box elimination rule. 
\end{enumerate}
\end{prop}

We show that each existence property yields the corresponding disjunction property.

\begin{prop}\label{EPtoDP}
Let $T$ be any $\LB$-theory. 
\begin{enumerate}
	\item If $T$ is an extension of $\PA(\K)$ and $T$ has $\MEP$ (resp.~$\B$-$\EP$), then $T$ has $\MDP$ (resp.~$\B$-$\DP$); 
	\item If $T$ has $\DB$-$\EP$ \textup{(}resp.~$\SB$-$\EP$\textup{)}, then $T$ has $\DB$-$\DP$ \textup{(}resp.~$\SB$-$\DP$\textup{)}. 
\end{enumerate}
\end{prop}
\begin{proof}
We only give a proof of Clause 1 for $\MEP$ and $\MDP$. 
Let $\varphi$ and $\psi$ be any sentences such that $T \vdash \Box \varphi \lor \Box \psi$. 
Then, $T \vdash \exists x \Box \bigl((x = 0 \land \varphi) \lor (x \neq 0 \land \psi) \bigr)$. 
By $\MEP$, there exists a natural number $n$ such that $T \vdash (\overline{n} = 0 \land \varphi) \lor (\overline{n} \neq 0 \land \psi)$. 
If $n=0$, $T \vdash \varphi$; if $n \neq 0$, $T \vdash \psi$. 
Therefore, $T$ has $\MDP$. 
\end{proof}

In the literature so far, modal disjunction and existence properties in modal arithmetic have been considered only for theories which are closed under the box elimination rule. 
As shown in Proposition \ref{DPEPb}, if $T$ is closed under the box elimination rule, then $\MDP$ and $\B$-$\DP$ are equivalent. 
Hence, $\MDP$ and $\B$-$\DP$ have often been identified in the literature.
Since the present paper also deals with theories that are not necessarily closed under the box elimination rule, we distinguish between $\MDP$ and $\B$-$\DP$.
In fact, as Figure \ref{Fig1} shows, there seems to be a large gap between the strength of these properties.

We explore nontrivial implications between $\DB$-$\DP$, $(\DB, \Sigma_1)$-$\DP$, $\DB$-$\DC$, $\B$-$\DP$, $(\B, \Sigma_1)$-$\DP$, and $\B$-$\DC$.

\begin{lem}\label{DPtoSDP}
Let $T$ be any r.e.~extension of $\PA(\K)$ having $\B^{n+1}$-$\DP$. 
Then, for any $\LB$-sentences $\varphi_1, \ldots, \varphi_n$ and $\Sigma_1$ sentence $\sigma$, if $T \vdash \Box \varphi_1 \lor \cdots \lor \Box \varphi_n \lor \sigma$, then $T \vdash \Box \varphi_i$ for some $i$ $(1 \leq i \leq n)$ or $T \vdash \sigma$. 
\end{lem}
\begin{proof}
Suppose $T \vdash \Box \varphi_1 \lor \cdots \lor \Box \varphi_n \lor \sigma$ and $T \nvdash \sigma$, and we would like to show $T \vdash \Box \varphi_i$ for some $i$. 
We may assume that $\sigma$ is of the form $\exists x \delta(x)$ for some $\Delta_0$ formula $\delta(x)$. 
Then, $\NS \models \forall x \neg \delta(x)$ because $T \nvdash \sigma$. 
By the Fixed Point Lemma, for each $i$ with $1 \leq i \leq n$, let $\psi_i^0$ and $\psi_i^1$ be $\Sigma_1$ sentences satisfying the following equivalences: 
\begin{itemize}
	\item $\PA \vdash \psi_i^0 \leftrightarrow \exists x \Bigl( \bigl(\delta(x) \lor \Prf_T(\gn{\Box \psi_i^1}, x) \bigr) \land \forall y < x \, \neg \Prf_T(\gn{\Box (\varphi_i \lor \psi_i^0)}, y) \Bigr)$; 
	\item $\PA \vdash \psi_i^1 \leftrightarrow \exists y \Bigl(\Prf_T(\gn{\Box (\varphi_i \lor \psi_i^0)}, y) \land \forall x \leq y \, \bigl(\neg \delta(x) \land \neg \Prf_T(\gn{\Box \psi_i^1}, x) \bigr) \Bigr)$.  
\end{itemize}
Then, for each $i$, we get $\PA \vdash \sigma \to \psi_i^0 \lor \psi_i^1$. 
Hence, we have 
\[
	\PA \vdash \sigma \to \psi_1^0 \lor \cdots \lor \psi_n^0 \lor (\psi_1^1 \land \cdots \land \psi_n^1).
\] 
By Theorem \ref{S1compl}, 
\[
	\PA(\K) \vdash \sigma \to \Box \psi_1^0 \lor \cdots \lor \Box \psi_n^0 \lor \Box (\psi_1^1 \land \cdots \land \psi_n^1).
\] 
Hence, 
\begin{equation}\label{imp1}
	\PA(\K) \vdash \sigma \to \Box (\varphi_1 \lor \psi_1^0) \lor \cdots \lor \Box (\varphi_n \lor \psi_n^0) \lor \Box (\psi_1^1 \land \cdots \land \psi_n^1).
\end{equation}
On the other hand, for each $i$, we have $\PA(\K) \vdash \Box \varphi_i \to \Box(\varphi_i \lor \psi_i^0)$. 
From our supposition, we obtain
\[
	T \vdash \Box (\varphi_1 \lor \psi_1^0) \lor \cdots \lor \Box (\varphi_n \lor \psi_n^0) \lor \sigma. 
\] 
By combining this with (\ref{imp1}), 
\[
	T \vdash \Box (\varphi_1 \lor \psi_1^0) \lor \cdots \lor \Box (\varphi_n \lor \psi_n^0) \lor \Box (\psi_1^1 \land \cdots \land \psi_n^1).
\] 
By $\B^{n+1}$-$\DP$, we have $T \vdash \Box (\varphi_i \lor \psi_i^0)$ for some $i$ or $T \vdash \Box (\psi_1^1 \land \cdots \land \psi_n^1)$. 
If $T \vdash \Box (\psi_1^1 \land \cdots \land \psi_n^1)$, then $T \vdash \Box \psi_i^1$ for each $i$. 
\begin{itemize}
	\item If $T \vdash \Box (\varphi_i \lor \psi_i^0)$ and $T \nvdash \Box \psi_i^1$, then $\NS \models \psi_i^1$ by the choice of $\psi_i^1$ because $\NS \models \forall x \neg \delta(x)$. 
	Thus, $T \vdash \psi_i^1$ by $\Sigma_1$-completeness, and hence $T \vdash \Box \psi_i^1$. 
	This is a contradiction. 
	\item If $T \vdash \Box \psi_i^1$ and $T \nvdash \Box(\varphi_i \lor \psi_i^0)$, then $\NS \models \psi_i^0$, and hence $T \vdash \psi_i^0$. 
	Thus, $T \vdash \varphi_i \lor \psi_i^0$ and hence $T \vdash \Box(\varphi_i \lor \psi_i^0)$, a contradiction. 
\end{itemize}
We have shown that in either case, for some $i$, both $\Box (\varphi_i \lor \psi_i^0)$ and $\Box \psi_i^1$ are provable in $T$.  
Since $\PA \vdash \psi_i^1 \to \neg \psi_i^0$, we have $T \vdash \Box \neg \psi_i^0$ for such an $i$. 
Therefore, we conclude $T \vdash \Box \varphi_i$. 
\end{proof}

From Propositions \ref{DBSent} and \ref{BDP} and Lemma \ref{DPtoSDP}, we obtain the following proposition. 

\begin{prop}\label{DPtoSDP2}
Let $T$ be any r.e.~extension of $\PA(\K)$. 
\begin{enumerate}
	\item If $T$ has $\DB$-$\DP$, then $T$ has $(\DB, \Sigma_1)$-$\DP$; 
	\item If $T$ has $\B$-$\DP$, then $T$ has $(\B, \Sigma_1)$-$\DP$. 
\end{enumerate}
\end{prop}

\begin{prop}\label{DPtoS1sound}
Let $T$ be any r.e.~extension of $\PA(\K)$ with $T \nvdash \Box \bot$. 
If $T$ has $(\B, \Sigma_1)$-$\DP$, then $T$ is $\Sigma_1$-sound. 
\end{prop}
\begin{proof}
We prove the contrapositive. 
Suppose that $T \nvdash \Box \bot$ and $T$ is not $\Sigma_1$-sound. 
Then, there exists a $\Delta_0$ formula $\delta(x)$ such that $T \vdash \exists x \delta(x)$ and $\NS \models \forall x \neg \delta(x)$. 
Let $\sigma_0$ and $\sigma_1$ be $\Sigma_1$ sentences satisfying the following equivalences: 
\begin{itemize}
	\item $\PA \vdash \sigma_0 \leftrightarrow \exists x \Bigl( \bigl(\delta(x) \lor \Prf_T(\gn{\sigma_1}, x) \bigr) \land \forall y < x \, \neg \Prf_T(\gn{\Box \sigma_0}, y) \Bigr)$; 
	\item $\PA \vdash \sigma_1 \leftrightarrow \exists y \Bigl(\Prf_T(\gn{\Box \sigma_0}, y) \land \forall x \leq y \, \bigl(\neg \delta(x) \land \neg \Prf_T(\gn{\sigma_1}, x) \bigr) \Bigr)$.  
\end{itemize}
Since $T \vdash \exists x \delta(x)$, we have $T \vdash \sigma_0 \lor \sigma_1$. 
Therefore, $T \vdash (\Box \sigma_0) \lor \sigma_1$ by Theorem \ref{S1compl}. 

Suppose, towards a contradiction, that $T \vdash \Box \sigma_0$ or $T \vdash \sigma_1$. 
Let $p$ be the smallest $T$-proof of $\Box \sigma_0$ or $\sigma_1$. 
If $p$ is a proof of $\Box \sigma_0$, then $\NS \models \sigma_1$ by the choice of $\sigma_1$. 
Hence, $T \vdash \sigma_1$ and thus $T \vdash \Box \sigma_1$. 
Since $T \vdash \sigma_0 \land \sigma_1 \to \bot$, we have $T \vdash \Box \bot$ because $T \vdash \Box \sigma_0 \land \Box \sigma_1$. 
This is a contradiction. 
If $p$ is a proof of $\sigma_1$, then it is shown $T \vdash \sigma_0$. 
This contradicts the consistency of $T$. 
Thus, we have shown that $T \nvdash \Box \sigma_0$ and $T \nvdash \sigma_1$. 
This means that $T$ does not have $(\B, \Sigma_1)$-$\DP$. 
\end{proof}

\begin{prop}\label{DPtoDC}
Let $T$ be any consistent r.e.~extension of $\PA(\K)$ with $T \nvdash \Box \bot$ and $\Gamma$ be a class of formulas with $\B \subseteq \Gamma$. 
If $T$ has $(\Gamma, \Sigma_1)$-$\DP$, then $T$ is $\Gamma$-$\DC$. 
\end{prop}
\begin{proof}
Suppose that $T$ has $(\Gamma, \Sigma_1)$-$\DP$. 
Let $\varphi$ be any $\Gamma$ sentence such that $T \vdash \varphi \lor \PR_T(\gn{\varphi})$. 
By $(\Gamma, \Sigma_1)$-$\DP$, $T \vdash \varphi$ or $T \vdash \PR_T(\gn{\varphi})$. 
Since $\B \subseteq \Gamma$, by Proposition \ref{DPtoS1sound}, $T$ is $\Sigma_1$-sound. 
Thus, in either case, we obtain $T \vdash \varphi$. 
\end{proof}

The converse implication also holds when $\Gamma$ is $\B$ or $\DB$. 
In order to prove this, we generalize the Fixed Point Lemma to modal arithmetic. 
It is proved by repeating a well-known proof, and so we omit it (see \cite{Boo93}).

\begin{lem}[The Fixed Point Lemma]\label{FPL}
For any $\LB$-formulas $\varphi_0(x_0, \ldots, x_{k-1})$, $\ldots$, $\varphi_{k-1}(x_0, \ldots, x_{k-1})$ with only the free variables $x_0, \ldots, x_{k-1}$, we can effectively find $\LB$-sentences $\psi_0, \ldots, \psi_{k-1}$ such that for each $i < k$, 
\[
	\PA_\Box \vdash \psi_i \leftrightarrow \varphi_i(\gn{\psi_0}, \ldots, \gn{\psi_{k-1}}). 
\]
Moreover, for each $i < k$, if $\varphi_i(x_0, \ldots, x_{k-1})$ is a $\SB$ formula, then such a $\psi_i$ can be found as a $\SB$ sentence. 
\end{lem}

\begin{prop}\label{DCtoSDP}
Let $T$ be any r.e.~extension of $\PA(\K)$. 
\begin{enumerate}
	\item If $T$ is $\DB$-$\DC$, then $T$ has $(\DB, \Sigma_1)$-$\DP$; 
	\item If $T$ is $\B$-$\DC$, then $T$ has $(\B, \Sigma_1)$-$\DP$. 
\end{enumerate}
\end{prop}
\begin{proof}
We prove only Clause 1. 
Clause 2 is proved similarly. 
Suppose that $T$ is $\DB$-$\DC$. 
Let $\varphi$ be any $\DB$ sentence and $\delta(x)$ be any $\Delta_0$ formula such that $T \vdash \varphi \lor \exists x \delta(x)$ and $T \nvdash \exists x \delta(x)$. 
We would like to show $T \vdash \varphi$. 
In this case, $\NS \models \forall x \neg \delta(x)$. 
By Proposition \ref{DBSent}, we may assume that $\varphi$ is of the form $\Box \psi_0 \lor \cdots \lor \Box \psi_{k-1}$. 
By the Fixed Point Lemma, let $\xi_0, \ldots, \xi_{k-1}$ be $\LB$-sentences satisfying the following equivalences for all $i < k$: 
\[
	\PA_\Box \vdash \xi_i \leftrightarrow \Bigl[\psi_i \lor \exists x \bigl(\delta(x) \land \forall y < x \, \neg \Prf_T(\gn{\Box \xi_0 \lor \cdots \lor \Box \xi_{k-1}}, y) \bigr)\Bigr].  
\]
Since $\PA_\Box \vdash \psi_i \to \xi_i$, we have $\PA(\K) \vdash \Box \psi_i \to \Box \xi_i$. 
Also, 
\begin{align*}
	\PA(\K) \vdash \exists x \delta(x) \land & \neg \PR_T(\gn{\Box \xi_0 \lor \cdots \lor \Box \xi_{k-1}}) \\
	& \ \ \ \to \exists x \bigl(\delta(x) \land \forall y < x \, \neg \Prf_T(\gn{\Box \xi_0 \lor \cdots \lor \Box \xi_{k-1}}, y) \bigr), \\
	& \ \ \ \to \PR_T \Bigl(\gn{\exists x \bigl(\delta(x) \land \forall y < x \, \neg \Prf_T(\gn{\Box \xi_0 \lor \cdots \lor \Box \xi_{k-1}}, y) \bigr)} \Bigr), \\
	& \ \ \ \to \PR_T(\gn{\xi_i}). 
\end{align*}
Since $\PA(\K) \vdash \PR_T(\gn{\xi_i}) \to \PR_T(\gn{\Box \xi_i})$ because $\PA$ can prove that the consequences of $T$ are closed under the rule \textsc{Nec}, we obtain 
$\PA(\K) \vdash \PR_T(\gn{\xi_i}) \to \PR_T(\gn{\Box \xi_0 \lor \cdots \lor \Box \xi_{k-1}})$. 
Thus, we have
\[
	\PA(\K) \vdash \bigl[\exists x \delta(x) \land \neg \PR_T(\gn{\Box \xi_0 \lor \cdots \lor \Box \xi_{k-1}}) \bigr] \to \PR_T(\gn{\Box \xi_0 \lor \cdots \lor \Box \xi_{k-1}}), 
\] 
and hence
\[
	\PA(\K) \vdash \exists x \delta(x) \to \PR_T(\gn{\Box \xi_0 \lor \cdots \lor \Box \xi_{k-1}}).
\] 
Then, by combining this with our assumption that $T \vdash \Box \psi_0 \lor \cdots \lor \Box \psi_{k-1} \lor \exists x \delta(x)$, we obtain 
\[
	T \vdash \Box \xi_0 \lor \cdots \lor \Box \xi_{k-1} \lor \PR_T(\gn{\Box \xi_0 \lor \cdots \lor \Box \xi_{k-1}}).
\] 
By $\DB$-$\DC$, we have 
\begin{equation}\label{dis1}
	T \vdash \Box \xi_0 \lor \cdots \lor \Box \xi_{k-1}.
\end{equation} 
Since 
\[
	\NS \models \exists y \bigl(\Prf_T(\gn{\Box \xi_0 \lor \cdots \lor \Box \xi_{k-1}}, y) \land \forall x \leq y \, \neg \delta(x) \bigr),
\]
this sentence is provable in $\PA(\K)$. 
Thus, 
\[
	\PA(\K) \vdash \neg \exists x \bigl(\delta(x) \land \forall y < x \, \neg \Prf_T(\gn{\Box \xi_0 \lor \cdots \lor \Box \xi_{k-1}}, y) \bigr). 
\]
Then, by the choice of $\xi_i$, $\PA(\K) \vdash \xi_i \to \psi_i$ for each $i$. 
From (\ref{dis1}), we conclude $T \vdash \Box \psi_0 \lor \cdots \lor \Box \psi_{k-1}$ and hence $T \vdash \varphi$. 
\end{proof}

In the statements of Propositions \ref{DPtoS1sound} and \ref{DPtoDC}, the condition ``$T \nvdash \Box \bot$'' is assumed. 
On the other hand, for consistent theories $T$ with $T \vdash \Box \bot$, the situation changes. 
Indeed, every $\B$ formula is provable in such a theory $T$. 
Thus, $T$ does not have $\MDP$ and $\MEP$. 
Also, every $\DB$ formula is $T$-provably equivalent to some $\Delta_0$ formula. 
Moreover, every $\DB$ sentence $\varphi$ is either provable or refutable in $T$. 
Therefore, we obtain the following proposition.

\begin{prop}\label{Verprov}
Let $T$ be any extension of $\PA(\Ver)$. 
Then, $T$ has $\DB$-$\DP$, $(\DB, \Sigma_1)$-$\DP$, and $\B$-$\EP$. 
Also, $T$ is $\B$-$\DC$. 
\end{prop}

From Propositions \ref{DPtoDC}, \ref{DCtoSDP}, and \ref{Verprov}, we have: 

\begin{cor}\label{CorB-DB}
Let $T$ be any consistent r.e.~extension of $\PA(\K)$. 
\begin{enumerate}
	\item If $T \nvdash \Box \bot$, then $T$ has $(\DB, \Sigma_1)$-$\DP$ if and only if $T$ is $\DB$-$\DC$;  
	\item $T$ has $(\B, \Sigma_1)$-$\DP$ if and only if $T$ is $\B$-$\DC$.  
\end{enumerate}
\end{cor}

For consistent r.e.~extensions of $\PA(\Ver)$, $\DB$-$\DC$ is strictly weaker than $\Sigma_1$-soundness. 

\begin{prop}\label{nCon}
Let $T$ be any consistent r.e.~extension of $\PA(\Ver)$. 
Then, the following are equivalent:
\begin{enumerate}
	\item $T$ is $\DB$-$\DC$. 
	\item $T \nvdash \neg \Con_T$. 
\end{enumerate}
\end{prop}
\begin{proof}
$(1 \Rightarrow 2)$: 
Suppose that $T$ is $\DB$-$\DC$. 
Since $T \nvdash \bot$, we obtain $T \nvdash \PR_T(\gn{\bot}) \lor \bot$. 
Hence, $T \nvdash \neg \Con_T$. 

$(2 \Rightarrow 1)$: 
Suppose $T \nvdash \neg \Con_T$. 
Let $\varphi$ be any $\DB$ sentence with $T \vdash \PR_T(\gn{\varphi}) \lor \varphi$. 
If $T \vdash \neg \varphi$, then $\varphi$ is $T$-equivalent to $\bot$. 
We have $T \vdash \PR_T(\gn{\bot}) \lor \bot$, and so $T \vdash \neg \Con_T$. 
This is a contradiction. 
Therefore, $T \vdash \varphi$. 
\end{proof}

There are $\mathcal{L}_A$-sound theories that do not have even $(\B, \Sigma_1)$-$\DP$.

\begin{prop}\label{nDP}\leavevmode
\begin{enumerate}
	\item $\PA(\Triv)$ does not have $(\B, \Sigma_1)$-$\DP$. 
	\item Let $T$ be any r.e.~theory such that $\PA(\Ver) \vdash T \vdash \PA(\KF)$ and let $U : = T + \{\PR_T(\gn{\Box \bot}) \lor \Box \bot\}$. 
If $T \nvdash \Box \bot$, then $U$ is $\mathcal{L}_A$-sound but is not $\B$-$\DC$. 
\end{enumerate}
\end{prop}
\begin{proof}
1. Let $\varphi$ be a $\Pi_1$ G\"odel sentence of $\PA$. 
Since $\PA \vdash \varphi \lor \neg \varphi$, we have $\PA(\Triv) \vdash \Box \varphi \lor \neg \varphi$. 
Since $\PA(\Triv)$ is a conservative extension of $\PA$ (Corollary \ref{Triv}), $\PA(\Triv) \nvdash \varphi$ and $\PA(\Triv) \nvdash \neg \varphi$. 
Then, $\PA(\Triv) \nvdash \Box \varphi$ and $\PA(\Triv) \nvdash \neg \varphi$. 

2. Since $U$ is a subtheory of $\PA(\Ver)$, $U$ is $\mathcal{L}_A$-sound by Corollary \ref{Ver}. 
Since $U \vdash \PR_{T}(\gn{\Box \bot}) \lor \Box \bot$, we have $U \vdash \PR_U(\gn{\Box \bot}) \lor \Box \bot$. 
Suppose, towards a contradiction, $U \vdash \Box \bot$. 
Since $\PR_{T}(\gn{\Box \bot}) \lor \Box \bot$ is a $\SB$ sentence, $T \vdash \PR_{T}(\gn{\Box \bot}) \lor \Box \bot \to \Box \bot$ by the $\SB$-deduction theorem (Corollary \ref{SBDT}). 
In particular, $T \vdash \PR_{T}(\gn{\Box \bot}) \to \Box \bot$. 
By L\"ob's theorem, $T \vdash \Box \bot$. 
This is a contradiction. 
Therefore, $U \nvdash \Box \bot$. 
Hence, $U$ is not $\B$-$\DC$. 
\end{proof}

\section{$\SB$-$\DP$ and related properties}\label{Sec:SB}

First of all, we consider the case that $T$ proves $\Box \bot$. 

\begin{prop}\label{SBDPEP}
Let $T$ be any consistent r.e.~extension of $\PA(\Ver)$. 
Then, the following are equivalent:
\begin{enumerate}
	\item $T$ is $\Sigma_1$-sound.
	\item $T$ has $\SB$-$\DP$. 
	\item $T$ has $(\SB, \Sigma_1)$-$\DP$. 
	\item $T$ has $\SB$-$\EP$. 
	\item $T$ has $\DB$-$\EP$. 
	\item $T$ is $\SB$-$\DC$.	
\end{enumerate}
\end{prop}
\begin{proof}
Since every $\SB$ (resp.~$\DB$) formula is $T$-provably equivalent to some $\Sigma_1$ (resp.~$\Delta_0$) formula, we have $(2 \Leftrightarrow 3)$. 
Also by Guaspari's theorem (\cite{Gua79}) on the equivalence of the $\Sigma_1$-soundness and $\Sigma_1$-$\DP$, the equivalence $(1 \Leftrightarrow 2)$ holds. 
Moreover, the implications ``$\Sigma_1$-sound $\Rightarrow \Sigma_1$-$\EP$'', ``$\Sigma_1$-$\EP \Rightarrow \Delta_0$-$\EP$'', and ``$\Delta_0$-$\EP \Rightarrow \Sigma_1$-sound'' are easily verified, we obtain that Clauses 1, 4, and 5 are pairwise equivalent. 
Finally, since the equivalence of the $\Sigma_1$-soundness and $\Sigma_1$-$\DC$ is shown in (\cite{Kur18}), we get $(1 \Leftrightarrow 6)$. 
\end{proof}

\begin{cor}\label{VerEP}
$\PA(\Ver)$ has $\SB$-$\EP$ but does not have $\MDP$. 
\end{cor}
\begin{proof}
Since $\PA(\Ver)$ is $\Sigma_1$-sound by Corollary \ref{Ver}, $\PA(\Ver)$ has $\SB$-$\EP$. 
On the other hand, $\PA(\Ver) \vdash \Box \bot$ and $\PA(\Ver) \nvdash \bot$, and thus $\PA(\Ver)$ does not have $\MDP$. 
\end{proof}

We then discuss theories in which $\Box \bot$ is not necessarily provable. 
Unlike the cases of $\DB$ and $\B$ (Proposition \ref{DPtoSDP2}), $(\SB, \Sigma_1)$-$\DP$ directly follows from $\SB$-$\DP$ because $\Sigma_1 \subseteq \SB$. 
Also, as in the cases of $\DB$ and $\B$ (Proposition \ref{DCtoSDP}), we obtain the following proposition. 

\begin{prop}\label{SBDC}
Let $T$ be any r.e.~$\LB$-theory extending $\PA$. 
If $T$ is $\SB$-$\DC$, then $T$ has $(\SB, \Sigma_1)$-$\DP$. 
\end{prop}
\begin{proof}
Suppose that $T$ is $\SB$-$\DC$. 
Let $\varphi$ be any $\SB$ sentence and $\delta(x)$ be any $\Delta_0$ formula such that $T \vdash \varphi \lor \exists x \delta(x)$ and $T \nvdash \exists x \delta(x)$. 
We would like to show $T \vdash \varphi$. 
In this case, $\NS \models \forall x \neg \delta(x)$. 
Let $\sigma$ be a $\Sigma_1$ sentence satisfying
\begin{equation}\label{eqSBDC}
	\PA \vdash \sigma \leftrightarrow \exists x \bigl(\delta(x) \land \forall y < x \, \neg \Prf_T(\gn{\varphi \lor \sigma}, y) \bigr).
\end{equation}
Since $T$ is an extension of $\PA$ and $\sigma$ is $\Sigma_1$, we have $\PA \vdash \sigma \to \PR_T(\gn{\sigma})$, and hence $\PA \vdash \sigma \to \PR_T(\gn{\varphi \lor \sigma})$. 
By the equivalence (\ref{eqSBDC}), we obtain 
\[
	\PA \vdash \exists x \delta(x) \land \neg \Prf_T(\gn{\varphi \lor \sigma}) \to \sigma.
\]
It follows $\PA \vdash \exists x \delta(x) \land \neg \Prf_T(\gn{\varphi \lor \sigma}) \to \PR_T(\gn{\varphi \lor \sigma})$, and hence
\[
	\PA \vdash \exists x \delta(x) \to \Prf_T(\gn{\varphi \lor \sigma}).
\]
Since $T \vdash \varphi \lor \exists x \delta(x)$, we obtain $T \vdash (\varphi \lor \sigma) \lor \PR_T(\gn{\varphi \lor \sigma})$. 
Since $\varphi \lor \sigma$ is a $\SB$ sentence, by $\SB$-$\DC$, we have $T \vdash \varphi \lor \sigma$. 
Since $\NS \models \forall x \neg \delta(x)$, $\NS \models \exists y \bigl(\Prf_T(\gn{\varphi \lor \sigma}, y) \land \forall x \leq y \, \neg \delta(x) \bigr)$ and this is provable in $T$. 
Then, $T \vdash \neg \sigma$, and thus $T \vdash \varphi$. 
\end{proof}

From Propositions \ref{DPtoDC}, \ref{SBDPEP}, and \ref{SBDC}, we obtain the following corollary. 

\begin{cor}\label{SBDPDC}
For any r.e.~extension $T$ of $\PA(\K)$, $T$ has $(\SB, \Sigma_1)$-$\DP$ if and only if $T$ is $\SB$-$\DC$. 
\end{cor}

Before proving our main theorem of this section, we prepare some notations and lemmas. 

\begin{definition}
For each $\SB$ formula $\varphi$, we define the $\LB$-formula $\varphi^-$ inductively as follows: 
\begin{enumerate}
	\item If $\varphi$ is $\Sigma_1$, then $\varphi^- : \equiv \varphi$; 
	\item If $\varphi$ is of the form $\Box \psi$, then $\varphi^- : \equiv \psi$; 
	\item Otherwise if $\varphi$ is of the form $\psi \land \sigma$, $\psi \lor \sigma$, $\exists x \psi$ or $\forall x < t\, \psi$, then $\varphi^-$ is respectively $\psi^- \land \sigma^-$, $\psi^- \lor \sigma^-$, $\exists x \psi^-$ or $\forall x < t\, \psi^-$. 
\end{enumerate}
\end{definition}

The operation $(\cdot)^-$ removes the outermost $\Box$ of nested occurrences of $\Box$'s in the formula. 
For example, $(\Box \Box \varphi \lor \Box \psi)^-$ is $\Box \varphi \lor \psi$. 
The following lemma is a strengthening of Theorem \ref{S1compl}. 

\begin{lem}\label{Minus}
For any $\SB$ formula $\varphi$, $\PA(\K) \vdash \varphi \to \Box (\varphi^-)$. 
\end{lem}
\begin{proof}
This lemma is proved by induction on the construction of $\varphi$ as in the proof of Theorem \ref{SBcompl}. 
Notice that if $\varphi$ is of the form $\Box \psi$, then $\varphi^- \equiv \psi$ and thus $\PA(\K) \vdash \varphi \to \Box (\varphi^-)$ holds. 
\end{proof}

\begin{definition}
For each $\DB$ formula $\varphi(\vec{x})$, we define the $\DB$ formula $\varphi^*(\vec{x}, \vec{v})$ with zero or more additional free variables $\vec{v}$ which do not occur in $\varphi(\vec{x})$ inductively as follows: 
\begin{enumerate}
	\item If $\varphi(\vec{x})$ is either $\Delta_0$ or of the form $\Box \psi(\vec{x})$, then $\varphi^*(\vec{x}) : \equiv \varphi(\vec{x})$; 
	
	\item Otherwise if $\varphi(\vec{x})$ is of the form $\psi(\vec{x}) \land \sigma(\vec{x})$, then $\varphi^*(\vec{x}, \vec{u}, \vec{v}) : \equiv \psi^*(\vec{x}, \vec{u}) \land \sigma^*(\vec{x}, \vec{v})$ where $\vec{u}$ and $\vec{v}$ are pairwise disjoint; 
	
	\item Otherwise if $\varphi(\vec{x})$ is of the form $\psi(\vec{x}) \lor \sigma(\vec{x})$, then 
	\[
		\varphi^*(\vec{x}, \vec{u}, \vec{v}, w) : \equiv \Bigl[ \bigl(w = 0 \land \psi^*(\vec{x}, \vec{u}) \bigr) \lor \bigl(w \neq 0 \land \sigma^*(\vec{x}, \vec{v}) \bigr) \Bigr]
	\]
	where $\vec{u}$, $\vec{v}$, and $w$ are pairwise disjoint; 
		
	\item Otherwise if $\varphi(\vec{x})$ is of the form $\forall y < t\, \psi(\vec{x}, y)$, then $\varphi^*(\vec{x}, \vec{v}) : \equiv \forall y < t\, \psi^*(\vec{x}, y, \vec{v})$; 
	
	\item Otherwise if $\varphi(\vec{x})$ is of the form $\exists y < t\, \psi(\vec{x}, y)$, then $\varphi^*(\vec{x}, \vec{v}, w)$ is the formula $\exists y < t\, (y = w \land \psi^*(\vec{x}, y, \vec{v}))$. 
\end{enumerate}
\end{definition}

From the definition, we can easily prove the following lemma by induction on the construction of $\varphi(\vec{x}) \in \DB$. 

\begin{lem}\label{Ast1}
For any $\DB$ formula $\varphi(\vec{x})$, $\PA_\Box \vdash \varphi(\vec{x}) \leftrightarrow \exists \vec{v} \varphi^*(\vec{x}, \vec{v})$. 
\end{lem}

The following lemma is an important feature of our two transformations $-$ and $*$.

\begin{lem}\label{Ast2}
Let $T$ be any r.e.~extension of $\PA(\K)$ such that $T \nvdash \Box \bot$. 
For any $\DB$ sentence $\varphi$, if there exist numbers $\vec{p}$ such that $T \vdash \Box (\varphi^*)^-(\vec{\overline{p}})$, then $T \vdash \varphi$. 
\end{lem}
\begin{proof}
We prove the lemma by induction on the construction of $\varphi$. 
\begin{itemize}
	\item If $\varphi$ is a $\Delta_0$ sentence, then $(\varphi^*)^- \equiv \varphi^- \equiv \varphi$. 
	Suppose $T \vdash \Box (\varphi^*)^-$, i.e., $T \vdash \Box \varphi$. 
	If $\NS \models \neg \varphi$, then $T \vdash \neg \varphi$ and $T \vdash \Box \neg \varphi$. 
	We have $T \vdash \Box \bot$, a contradiction. 
	Hence, $\NS \models \varphi$. 
	We conclude $T \vdash \varphi$. 
	
	\item If $\varphi$ is of the form $\Box \psi$, then $(\varphi^*)^- \equiv (\Box \psi)^- \equiv \psi$. 
	Suppose $T \vdash \Box (\varphi^*)^-$. 
	Then, $T \vdash \varphi$. 
	
	\item If $\varphi$ is of the form $\psi \land \sigma$, then $(\varphi^*)^-(\vec{u}, \vec{v}) \equiv (\psi^*)^-(\vec{u}) \land (\sigma^*)^-(\vec{v})$. 
	Suppose $T \vdash \Box (\varphi^*)^-(\vec{\overline{p}}, \vec{\overline{q}})$. 
	Then, $T \vdash \Box (\psi^*)^-(\vec{\overline{p}})$ and $T \vdash \Box (\sigma^*)^-(\vec{\overline{q}})$. 
	By the induction hypothesis, $T \vdash \psi$ and $T \vdash \sigma$. 
	We conclude $T \vdash \varphi$. 
	
	\item If $\varphi$ is of the form $\psi \lor \sigma$, then 
	\[
		(\varphi^*)^-(\vec{u}, \vec{v}, w) \equiv \Bigl[ \bigl(w = 0 \land (\psi^*)^-(\vec{u}) \bigr) \lor \bigl(w \neq 0 \land (\sigma^*)^-(\vec{v}) \bigr) \Bigr].
	\] 
	Suppose $T \vdash \Box (\varphi^*)^-(\vec{\overline{p}}, \vec{\overline{q}}, \overline{r})$. 
	Then, 
	\[
		T \vdash \Box \Bigl[ \bigl(\overline{r} = 0 \land (\psi^*)^-(\vec{\overline{p}}) \bigr) \lor \bigl(\overline{r} \neq 0 \land (\sigma^*)^-(\vec{\overline{q}}) \bigr) \Bigr]. 
	\]
	If $r = 0$, then $T \vdash \Box (\psi^*)^-(\vec{\overline{p}})$. 
	By the induction hypothesis, $T \vdash \psi$. 
	If $r \neq 0$, then $T \vdash \Box (\sigma^*)^-(\vec{\overline{q}})$. 
	By the induction hypothesis, $T \vdash \sigma$. 
	In either case, $T \vdash \varphi$. 
	
	\item If $\varphi$ is of the form $\forall x < t\, \psi(x)$ for some $\mathcal{L}_A$-term $t$, then $(\varphi^*)^-(\vec{v})$ is the formula $\forall x < t\, (\psi^*)^-(\vec{v}, x)$. 
	Since $\varphi$ is a sentence, $t$ is a closed term. 
	Let $k$ be the value of the term $t$ and suppose $T \vdash \Box (\varphi^*)^-(\vec{\overline{p}})$. 
	Then, for all $n < k$, $T \vdash \Box (\psi^*)^-(\vec{\overline{p}}, \overline{n})$. 
	By the induction hypothesis, $T \vdash \psi(\overline{n})$. 
	We obtain $T \vdash \varphi$. 
	
	\item If $\varphi$ is of the form $\exists x < t\, \psi(x)$ for some closed term $t$, then 
	\[
		(\varphi^*)^-(\vec{v}, w) \equiv \exists x < t \bigl(x = w \land (\psi^*)^-(\vec{v}, x) \bigr).
	\] 
	Suppose $T \vdash \Box (\varphi^*)^-(\vec{\overline{p}}, \overline{q})$. 
	Then, $T \vdash \Box \exists x < t  \bigl(x = \overline{q} \land (\psi^*)^-(\vec{\overline{p}}, x) \bigr)$. 
	Since $T \nvdash \Box \bot$, the value of $t$ is larger than $q$. 
	Since $T \vdash \Box (\psi^*)^-(\vec{\overline{p}}, \overline{q})$, by the induction hypothesis, $T \vdash \psi(\overline{q})$. 
	Then, $T \vdash \exists x < t\, \psi(x)$, that is, $T \vdash \varphi$. 
\end{itemize}
\end{proof}

We are ready to prove our main theorem of this section.

\begin{thm}\label{SBDP}
Let $T$ be any r.e.~extension of $\PA(\K)$ such that $T \nvdash \Box \bot$. 
Then, the following are equivalent: 
\begin{enumerate}
	\item $T$ has $\DB$-$\DP$ and $(\SB, \Sigma_1)$-$\DP$. 	
	\item $T$ has $\B$-$\EP$. 
	\item $T$ has $\SB$-$\EP$.
	\item $T$ has $\SB$-$\DP$. 
\end{enumerate}
\end{thm}
\begin{proof}
$(1 \Rightarrow 2)$: 
Let $\varphi(x)$ be any $\LB$-formula with no free variables except possibly $x$, such that $T \vdash \exists x \Box \varphi(x)$. 
By the Fixed Point Lemma, let $\psi$ be a $\SB$ sentence satisfying
\[
	\PA_\Box \vdash \psi \leftrightarrow \exists x \bigl(\Box \varphi(x) \land \forall y < x \, \neg \Prf_T(\gn{\psi}, y) \bigr). 
\]
Since $\PA_\Box \vdash \exists x \Box \varphi(x) \land \neg \PR_T(\gn{\psi}) \to \psi$, we have $T \vdash \PR_T(\gn{\psi}) \lor \psi$. 
Since $T$ is $\SB$-$\DC$ by Corollary \ref{SBDPDC}, we obtain $T \vdash \psi$. 
By the choice of $\psi$, 
\begin{equation}\label{eq1}
	T \vdash \exists x \bigl(\Box \varphi(x) \land \forall y < x \, \neg \Prf_T(\gn{\psi}, y) \bigr).
\end{equation} 
Let $p$ be a proof of $\psi$ in $T$, then $T \vdash \Prf_T(\gn{\psi}, \overline{p})$ and thus $T \vdash \exists x \leq \overline{p} \, \Box \varphi(x)$ by (\ref{eq1}). 
Then, $T \vdash \Box \varphi(\overline{0}) \lor \cdots \lor \Box \varphi(\overline{p})$. 
Since $T$ has $\B^{p+1}$-$\DP$ by Proposition \ref{BDP}.2, there exists $k \leq p$ such that $T \vdash \Box \varphi(\overline{k})$. 
Therefore, $T$ has $\B$-$\EP$. 

$(2 \Rightarrow 3)$: 
Let $\varphi(x)$ be any $\SB$ formula without having free variables except $x$ such that $T \vdash \exists x \varphi(x)$. 
By Proposition \ref{MSandBD}, there exists a $\DB$ formula $\psi(x, y)$ such that $\PA_\Box \vdash \varphi(x) \leftrightarrow \exists y \psi(x, y)$. 
Also, by Lemma \ref{Ast1}, $\PA_\Box \vdash \psi(x, y) \leftrightarrow \exists \vec{v} \psi^*(x, y, \vec{v})$. 
Then, $T \vdash \exists x \exists y \exists \vec{v} \psi^*(x, y, \vec{v})$ and 
\[
	T \vdash \exists w \, \exists x \leq w \, \exists y \leq w \, \exists \vec{v} \leq w\,  \bigl( w = \langle x, y, \vec{v} \rangle \land \psi^*(x, y, \vec{v}) \bigr).
\]
Here $\langle x, y, \vec{v} \rangle$ is an appropriate iteration of usual $\Delta_0$ representable bijective pairing function $\langle \cdot, \cdot \rangle$. 
We may assume that $\PA$ proves $x\, {\leq}\, \langle x, y \rangle$ and $y \, {\leq} \, \langle x, y \rangle$. 
By Lemma \ref{Minus}, 
\[
	T \vdash \exists w \, \Box \exists x \leq w \, \exists y \leq w \, \exists \vec{v} \leq w \, \bigl(w = \langle x, y, \vec{v} \rangle \land (\psi^*)^-(x, y, \vec{v}) \bigr).
\]
By $\B$-$\EP$, there exists a natural number $k$ such that
\[
	T \vdash \Box \exists x \leq \overline{k} \, \exists y \leq \overline{k} \, \exists \vec{v} \leq \overline{k} \, \bigl(\overline{k} = \langle x, y, \vec{v} \rangle \land (\psi^*)^-(x, y, \vec{v}) \bigr).
\]
For the unique $p$, $q$, and $\vec{r}$ such that $k = \langle p, q, \vec{r} \rangle$, 
\[
	T \vdash \Box (\psi^*)^-(\overline{p}, \overline{q}, \vec{\overline{r}}).
\]
By Lemma \ref{Ast2}, we obtain $T \vdash \psi(\overline{p}, \overline{q})$. 
Then, $T \vdash \varphi(\overline{p})$. 
Therefore, $T$ has $\SB$-$\EP$. 

$(3 \Rightarrow 4)$: 
By Proposition \ref{EPtoDP}.2. 

$(4 \Rightarrow 1)$:
This is trivial. 
\end{proof}

In order to derive the equivalence of $\MDP$ and $\MEP$ from Theorem \ref{SBDP}, we prove a proposition that connects $\MDP$ (resp.~$\MEP$) and $\SB$-$\DP$ (resp.~$\SB$-$\EP$). 

\begin{prop}\label{MDPtoSBDP}
Let $T$ be any r.e.~extension of $\PA(\KF)$. 
\begin{enumerate}
	\item If $T$ has $\MDP$, then $T$ also has $\SB$-$\DP$;  
	\item If $T$ has $\MEP$, then $T$ also has $\SB$-$\EP$. 
\end{enumerate}
\end{prop}
\begin{proof}
1. Let $\varphi$ and $\psi$ be any $\SB$ sentences such that $T \vdash \varphi \lor \psi$. 
By Theorem \ref{SBcompl}, $T \vdash \Box \varphi \lor \Box \psi$. 
By $\MDP$, we obtain $T \vdash \varphi$ or $T \vdash \psi$. 
Therefore, $T$ has $\SB$-$\DP$. 

Clause 2 is proved similarly. 
\end{proof}


\begin{cor}\label{FSrev}
For any r.e.~extension $T$ of $\PA(\KF)$, $T$ has $\MDP$ if and only if $T$ has $\MEP$. 
\end{cor}
\begin{proof}
Since $\MEP$ implies $\MDP$ by Proposition \ref{EPtoDP}.1, it suffices to show that $\MDP$ implies $\MEP$. 
We may assume that $T$ is consistent. 
If $T$ has $\MDP$, then $T$ has $\SB$-$\DP$ by Proposition \ref{MDPtoSBDP}. 
Also, $T$ is closed under the box elimination rule by Proposition \ref{DPEPb}. 
Then, $T \nvdash \Box \bot$ by the consistency of $T$. 
By Theorem \ref{SBDP}, $T$ has $\B$-$\EP$. 
By Proposition \ref{DPEPb} again, we conclude that $T$ has $\MEP$. 
\end{proof}

\begin{remark}\label{RemFS}
In the introduction, we imprecisely mentioned the result of Friedman and Sheard (\cite{FS89}). 
Firstly, Friedman and Sheard actually proved their theorem in the setting where the use of the rule \textsc{Nec} and the axiom $\forall \vec{x}(\Box(\varphi \to \psi) \to (\Box \varphi \to \Box \psi))$ is restricted, that is, in the non-normal setting. 
In our normal setting, the following result follows from their theorem: 
For any r.e.~extension $T$ of $\PA(\KF) + \{\forall \vec{x} \Box (\Box \varphi \to \varphi) \mid \varphi \in \Delta_0\}$, if $T$ is closed under the box elimination rule, then $T$ has $\B$-$\DP$ if and only if $T$ has $\B$-$\EP$. 
Then, in the light of Proposition \ref{DPEPb}, this statement can be rewritten as follows: 
For any r.e.~extension $T$ of $\PA(\KF) + \{\forall \vec{x} \Box (\Box \varphi \to \varphi) \mid \varphi \in \Delta_0\}$, $T$ has $\MDP$ if and only if $T$ has $\MEP$. 
Therefore, our Corollary \ref{FSrev} shows that the same consequence is obtained without using the axiom scheme $\{\forall \vec{x} \Box (\Box \varphi \to \varphi) \mid \varphi \in \Delta_0\}$. 
Notice that by Theorem \ref{S1compl}, over $\PA(\K)$, $\{\forall \vec{x} \Box (\Box \varphi \to \varphi) \mid \varphi \in \Delta_0\}$ is equivalent to a single sentence $\Box \neg \Box \bot$. 
\end{remark}

\section{Generalizations of the notions of soundness and $\Sigma_1$-soundness}\label{Sec:Soundness}

In this section, we introduce several notions related to the soundness of theories of modal arithmetic with respect to $\LB$-sentences. 
This section consists of three subsections. 
In the first subsection, we introduce the notion of the $\LB$-soundness, and prove several $\LB$-theories are actually $\LB$-sound. 
In the second subsection, we introduce the notions of the $\SB$-soundness and the weak $\SB$-soundness. 
Then, we prove that over appropriate theories, the $\SB$-soundness and the weak $\SB$-soundness characterize $\MDP$ and $\SB$-$\DP$, respectively. 
In the last subsection, we prove two non-implications between the properties as applications of the results we have obtained so far.

\subsection{$\LB$-soundness}

We formulate the notion of the $\LB$-soundness under the interpretation that boxed formulas represent the provability of some formula in the standard model $\NS$ of arithmetic.
To do so, we once translate each $\LB$-sentence into an $\mathcal{L}_A$-sentence using a provability predicate, and then consider the truth of the translated sentence in $\NS$. 
First, we introduce two types of translations $\pi_T$ and $\pi'_T$. 

\begin{definition}[$\pi$-translations]
Let $T$ be any r.e.~$\LB$-theory. 
We define a translation $\pi_T$ of $\LB$-formulas into $\mathcal{L}_A$-formulas inductively as follows: 
\begin{enumerate}
	\item If $\varphi$ is an $\mathcal{L}_A$-formula, then $\pi_T(\varphi) : \equiv \varphi$; 
	\item $\pi_T$ preserves logical connectives and quantifiers; 
	\item $\pi_T(\Box \varphi(\vec{x})) : \equiv \PR_T(\ulcorner \varphi(\vec{\dot{x}}) \urcorner)$. 
\end{enumerate}
\end{definition}

Here $\gn{\varphi(\vec{\dot{x}})}$ is a primitive recursive term corresponding to a primitive recursive function calculating the G\"odel number of $\varphi(\vec{\overline{n}})$ from $\vec{n}$. 
Note that $\vec{x}$ are free variables in the formula $\PR_T(\gn{\varphi(\vec{\dot{x}})})$. 

\begin{definition}[$\pi'$-translations]
Let $T$ be any r.e.~$\LB$-theory. 
We define a translation $\pi'_T$ of $\LB$-formulas into $\mathcal{L}_A$-formulas inductively as follows: 
\begin{enumerate}
	\item If $\varphi$ is an $\mathcal{L}_A$-formula, then $\pi'_T(\varphi) : \equiv \varphi$; 
	\item $\pi'_T$ preserves logical connectives and quantifiers; 
	\item $\pi'_T(\Box \varphi(\vec{x})) : \equiv \PR_T(\ulcorner \varphi(\vec{\dot{x}}) \urcorner) \land \pi'_T(\varphi(\vec{x}))$. 
\end{enumerate}
\end{definition}

The translation $\pi'_T$ is a formalization of Shapiro's slash interpretation (\cite{Sha85}), introduced in Halbach and Horsten (\cite{HH00}) under the name $\sigma_T$.

\begin{definition}
Let $T$ be any r.e.~$\LB$-theory.  
\begin{itemize}
	\item $T$ is said to be \textit{$\LB$-sound} if for any $\LB$-sentence $\varphi$, if $T \vdash \varphi$, then $\NS \models \pi_T(\varphi)$; 
	\item $T$ is said to be \textit{alternatively $\LB$-sound} if for any $\LB$-sentence $\varphi$, if $T \vdash \varphi$, then $\NS \models \pi'_T(\varphi)$. 
\end{itemize}
\end{definition}

Actually, these two notions are equivalent. 

\begin{prop}\label{soundness}
For any r.e.~$\LB$-theory $T$, the following are equivalent: 
\begin{enumerate}
	\item $T$ is $\LB$-sound. 
	\item $T$ is alternatively $\LB$-sound. 
\end{enumerate}
\end{prop}
\begin{proof}
$(1 \Rightarrow 2)$: 
Suppose that $T$ is $\LB$-sound. 
We prove by induction on the construction of $\varphi$ that for all $\LB$-sentences $\varphi$, $\NS \models \pi_T(\varphi) \leftrightarrow \pi_T'(\varphi)$. 
If $\varphi$ is an atomic $\mathcal{L}_A$-sentence, then $\pi_T(\varphi)$ coincides with $\pi'_T(\varphi)$. 
The cases for Boolean connectives are easy. 

If $\varphi$ is of the form $\forall x \psi(x)$, then for any natural number $n$, $\NS \models \pi_T(\psi(\overline{n})) \leftrightarrow \pi'_T(\psi(\overline{n}))$ by the induction hypothesis. 
Then, $\NS \models \forall x \bigl(\pi_T(\psi(x)) \leftrightarrow \pi'_T(\psi(x) \bigr)$ and hence $\NS \models \forall x \pi_T(\psi(x)) \leftrightarrow \forall x \pi'_T(\psi(x))$. 
This means $\NS \models \pi_T(\varphi) \leftrightarrow \pi'_T(\varphi)$. 

Suppose that $\varphi$ is of the form $\Box \psi$. 
Since $\pi_T(\Box \psi)$ is $\PR_T(\gn{\psi})$, by the $\LB$-soundness of $T$, 
$\NS \models \pi_T(\Box \psi)$ if and only if $\NS \models \PR_T(\gn{\psi}) \land \pi_T(\psi)$. 
By the induction hypothesis, $\NS \models \PR_T(\gn{\psi}) \land \pi_T(\psi)$ if and only if $\NS \models \PR_T(\gn{\psi}) \land \pi'_T(\psi)$. 
Thus, $\NS \models \pi_T(\Box \psi) \leftrightarrow \pi'_T(\Box \psi)$. 

$(2 \Rightarrow 1)$: 
Suppose that $T$ is alternatively $\LB$-sound. 
Similarly, we only prove that for all $\LB$-sentences $\psi$, $\NS \models \pi'_T(\Box \psi) \leftrightarrow \pi_T(\Box \psi)$. 
$\NS \models \pi'_T(\Box \psi)$ is equivalent to $\NS \models \PR_T(\gn{\psi}) \land \pi'_T(\psi)$. 
Then, by the alternative $\LB$-soundness of $T$, this is equivalent to $\NS \models \PR_T(\gn{\psi})$. 
This is exactly $\NS \models \pi_T(\Box \psi)$. 
\end{proof}

Here we show some propositions that help to prove the $\LB$-soundness of each $\LB$-theory. 

\begin{prop}\label{pi}
Let $T$ be an $\LB$-theory obtained by adding some axioms of $\PA(\GL)$ into $\PA_\Box$. 
For any $\LB$-formula $\varphi$, if $T \vdash \varphi$, then for any r.e.~extension $U$ of $T$, $\PA \vdash \pi_U(\varphi)$. 
\end{prop}
\begin{proof}
Let $U$ be any r.e.~extension of $T$. 
As in the proof of Proposition \ref{alpha}, by induction on the length of proofs of $\varphi$ in $T$, we prove that for any $\LB$-formula $\varphi$, if $T \vdash \varphi$, then $\PA \vdash \pi_U(\varphi)$. 
We only give proofs of the following four cases. 
\begin{itemize}
	\item If $\varphi$ is $\forall \vec{x} \bigl(\Box(\psi(\vec{x}) \to \sigma(\vec{x})) \to (\Box \psi(\vec{x}) \to \Box \sigma(\vec{x})) \bigr)$, then $\pi_U(\varphi)$ is 
\[
	\forall \vec{x} \bigl(\PR_U(\gn{\psi(\vec{\dot{x}}) \to \sigma(\vec{\dot{x}})}) \to (\PR_U(\gn{\psi(\vec{\dot{x}})}) \to \PR_U(\gn{\sigma(\vec{\dot{x}})})) \bigr),
\]
and this is provable in $\PA$. 

	\item If $\varphi$ is $\forall \vec{x} \bigl(\Box \psi(\vec{x}) \to \Box \Box \psi(\vec{x}) \bigr)$, then $\pi_U(\varphi)$ is 
\[
	\forall \vec{x} \bigl(\PR_U(\gn{\psi(\vec{\dot{x}})}) \to \PR_U(\gn{\Box \psi(\vec{\dot{x}})}) \bigr).
\] 
	Since $\PA$ proves the fact that the consequences of $U$ are closed under the rule \textsc{Nec}, this sentence is provable in $\PA$. 

	\item If $\varphi$ is $\forall \vec{x} \bigl(\Box (\Box \psi(\vec{x}) \to \psi(\vec{x})) \to \Box \psi(\vec{x}) \bigr)$, then we reason as follows: 
By invoking \textsc{Nec}, 
\[
	\PA \vdash \PR_U(\gn{\Box \psi(\vec{\dot{x}}) \to \psi(\vec{\dot{x}})}) \to \PR_U(\gn{\Box(\Box \psi(\vec{\dot{x}}) \to \psi(\vec{\dot{x}}))}). 
\]
Since $U$ is an extension of $T$, we have $\PA \vdash \PR_U(\gn{\varphi})$ and hence
\[
	\PA \vdash \PR_U(\gn{\Box(\Box \psi(\vec{\dot{x}}) \to \psi(\vec{\dot{x}}))}) \to \PR_U(\gn{\Box \psi(\vec{\dot{x}})}). 
\]
Then, 
\[
	\PA \vdash \PR_U(\gn{\Box \psi(\vec{\dot{x}}) \to \psi(\vec{\dot{x}})}) \to \PR_U(\gn{\Box \psi(\vec{\dot{x}})}), 
\]
and thus
\[
	\PA \vdash \PR_U(\gn{\Box \psi(\vec{\dot{x}}) \to \psi(\vec{\dot{x}})}) \to \PR_U(\gn{\psi(\vec{\dot{x}})}). 
\]
This means $\PA \vdash \pi_U(\varphi)$. 

	\item If $\varphi(\vec{x})$ is derived from $\psi(\vec{x})$ by \textsc{Nec}, then $\varphi(\vec{x}) \equiv \Box \psi(\vec{x})$. 
	Since $U \vdash \psi(\vec{x})$, $\PA \vdash \PR_U(\gn{\psi(\vec{\dot{x}})})$. 
	Thus, $\PA \vdash \pi_U(\varphi(\vec{x}))$. 
\end{itemize}
\end{proof}

\begin{prop}\label{ax-soundness1}
Let $T$ be an $\LB$-theory obtained by adding some axioms of $\PA(\GL)$ into $\PA_\Box$, and let $U$ be any r.e.~extension of $T$. 
If $\NS \models \pi_U(\varphi)$ for all $\varphi \in U \setminus T$, then $U$ is $\LB$-sound. 
\end{prop}
\begin{proof}
Suppose that $\NS \models \pi_U(\varphi)$ for all $\varphi \in U \setminus T$. 
We prove by induction on the length of a proof of $\varphi$ in $U$ that for all $\mathcal{L}_A$-formulas $\varphi$, if $U \vdash \varphi$, then $\NS \models \pi_U(\forall \vec{x} \varphi)$. 
\begin{itemize}
	\item If $\varphi$ is an axiom of $T$ or a logical axiom, then $\PA \vdash \pi_U(\forall \vec{x} \varphi)$ by Proposition \ref{pi}. 
	Thus, $\NS \models \pi_U(\forall \vec{x} \varphi)$. 
	\item If $\varphi$ is in $U \setminus T$, then $\NS \models \pi_U(\varphi)$ by the supposition. 
	\item If $\varphi$ is derived from $\psi$ and $\psi \to \varphi$ by MP, then by the induction hypothesis, $\NS \models \pi_U(\forall \vec{x} \psi)$ and $\NS \models \pi_U(\forall \vec{x}( \psi \to \varphi))$. 
	Then, $\NS \models \pi_U(\forall \vec{x} \varphi)$. 
	\item If $\varphi$ is derived from $\psi(y)$ by \textsc{Gen}, then $\varphi \equiv \forall y \psi(y)$. 
	By the induction hypothesis, $\NS \models \pi_U(\forall \vec{x} \forall y \psi(y))$. 
	Hence, $\NS \models \pi_U(\forall \vec{x} \varphi)$. 
	\item If $\varphi$ is derived from $\psi$ by \textsc{Nec}, then $\varphi \equiv \Box \psi$ and $U \vdash \psi$. 
	We have $\NS \models \forall \vec{x} \, \PR_U(\gn{\psi})$, and equivalently $\NS \models \pi_U(\forall \vec{x} \varphi)$. 
\end{itemize}
\end{proof}

\begin{cor}
The theories $\PA_\Box$, $\PA(\K)$, $\PA(\KF)$, and $\PA(\GL)$ are $\LB$-sound. 
\end{cor}

Here, we give some more examples of $\LB$-sound theories. 
Let $x \in W_y$ be a $\Sigma_1$ formula saying that $x$ is in the $y$-th r.e.~set. 
Reinhardt's Weak Mechanistic Thesis (WMT) is the following schema:  
\begin{itemize}
	\item $\exists y \forall x \bigl(\Box \varphi(x) \leftrightarrow x \in W_y \bigr)$, where $\varphi(x)$ is an $\LB$-formula having lone free variable $x$. 
\end{itemize}
When $\Box$ is interpreted as knowledge, WMT can be thought as a formalization of `Knowledge is mechanical'. 
Concerning WMT, we obtain the following corollary to Proposition \ref{ax-soundness1}. 

\begin{cor}
Let $T$ be an r.e.~$\LB$-theory obtained by adding some axioms of $\PA(\GL)$ into $\PA_\Box$. 
Then, the theory $U : = T + \mathrm{WMT}$ is $\LB$-sound. 
\end{cor}
\begin{proof}
Since $\PR_U(\gn{\varphi(\dot{x})})$ is a $\Sigma_1$ formula, there exists a natural number $e$ such that 
\[
	\NS \models \forall x \bigl(\PR_U(\gn{\varphi(\dot{x})}) \leftrightarrow x \in W_{\overline{e}} \bigr). 
\]
Then, we have $\NS \models \pi_U \Bigl(\exists y \forall x \bigl(\Box \varphi(x) \leftrightarrow x \in W_y \bigr) \Bigr)$. 
By Proposition \ref{ax-soundness1}, the theory $U$ is $\LB$-sound. 
\end{proof}

We prove an analogue of Proposition \ref{pi} with respect to $\pi'$-translations. 

\begin{prop}\label{pi'}
Let $T$ be an $\LB$-theory obtained by adding some axioms of $\PA(\SF)$ into $\PA_\Box$. 
For any $\LB$-formula $\varphi$, if $T \vdash \varphi$, then for any r.e.~extension $U$ of $T$, $\PA \vdash \pi'_U(\varphi)$. 
\end{prop}
\begin{proof}
Let $U$ be any r.e.~extension of $T$. 
As in the proof of Proposition \ref{alpha}, we prove by induction on the length of proofs of $\varphi$ in $T$ that for any $\LB$-formula $\varphi$, if $T \vdash \varphi$, then $\PA \vdash \pi'_U(\varphi)$. 
We only give proofs of the following four cases. 
\begin{itemize}
	\item If $\varphi$ is $\forall \vec{x} \bigl(\Box(\psi(\vec{x}) \to \sigma(\vec{x})) \to (\Box \psi(\vec{x}) \to \Box \sigma(\vec{x})) \bigr)$, then $\pi'_U(\varphi)$ is
\begin{align*}
	& \forall \vec{x} \Bigl(\PR_U(\gn{\psi(\vec{\dot{x}}) \to \sigma(\vec{\dot{x}})}) \land \pi'_U(\psi(\vec{x}) \to \sigma(\vec{x}))\\
	& \ \ \ \ \  \to \bigl(\bigl[\PR_U(\gn{\psi(\vec{\dot{x}})}) \land \pi'_U(\psi(\vec{x})) \bigr] \to \bigl[\PR_U(\gn{\sigma(\vec{\dot{x}})}) \land \pi'_U(\sigma(\vec{x})) \bigr] \bigr) \Bigr).
\end{align*}
	This sentence is provable in $\PA$. 

	\item If $\varphi$ is $\forall \vec{x}(\Box \psi(\vec{x}) \to \psi(\vec{x}))$, then $\pi'_U(\varphi)$ is
\[
	\forall \vec{x} \bigl(\PR_U(\gn{\psi(\vec{\dot{x}})}) \land \pi'_U(\psi(\vec{x})) \to \pi'_U(\psi(\vec{x})) \bigr), 
\]
and this is obviously provable in $\PA$. 

	\item If $\varphi$ is $\forall \vec{x}(\Box \psi(\vec{x}) \to \Box \Box \psi(\vec{x}))$, then $\pi_U(\varphi)$ is 
\[
	\forall \vec{x} \Bigl(\bigl[\PR_U(\gn{\psi(\vec{\dot{x}})}) \land \pi'_U(\psi(\vec{x})) \bigr] \to \bigl[\PR_U(\gn{\Box \psi(\vec{\dot{x}})}) \land \PR_U(\gn{\psi(\vec{\dot{x}})}) \land \pi'_U(\psi(\vec{x})) \bigr] \Bigr),
\]
and this is provable in $\PA$. 

	\item If $\varphi(\vec{x})$ is derived from $\psi(\vec{x})$ by \textsc{Nec}, then $\varphi(\vec{x}) \equiv \Box \psi(\vec{x})$. 
	Since $T \vdash \psi(\vec{x})$, by the induction hypothesis, $\PA \vdash \pi'_U(\psi(\vec{x}))$. 
	Also, $\PA \vdash \PR_U(\gn{\psi(\vec{\dot{x}})})$ because $U$ is an extension of $T$. 
	Thus, $\PA \vdash \pi'_U(\varphi(\vec{x}))$. 
\end{itemize}\end{proof}

As in the proof of Proposition \ref{ax-soundness1}, we can prove the following proposition from Propositions \ref{soundness} and \ref{pi'}. 

\begin{prop}[{cf.~\cite[TB]{Sha85}}]\label{ax-soundness2}
Let $T$ be an $\LB$-theory obtained by adding some axioms of $\PA(\SF)$ into $\PA_\Box$, and let $U$ be any r.e.~extension of $T$. 
If $\NS \models \pi'_U(\varphi)$ for all $\varphi \in U \setminus T$, then $U$ is $\LB$-sound. 
\end{prop}

\begin{cor}
The theories $\PA(\KT)$ and $\PA(\SF)$ are $\LB$-sound. 
\end{cor}

The alternative $\LB$-soundness of $\PA(\SF)$ is already proved by Shapiro \cite[TB']{Sha85}. 

\begin{cor}\label{egLB-sound}
Let $T$ be an r.e.~$\LB$-theory obtained by adding some axioms of $\PA(\SF) + \{\Box \Diamond \varphi \to \varphi \mid \varphi$ is an $\mathcal{L}_A$-sentence$\}$ into $\PA_\Box$. 
Then, $T$ is $\LB$-sound. 
\end{cor}
\begin{proof}
By Proposition \ref{ax-soundness2}, it suffices to show that for any $\mathcal{L}_A$-sentence $\varphi$, if $\Box \Diamond \varphi \to \varphi \in T$, then $\NS \models \pi'_T(\Box \Diamond \varphi \to \varphi)$. 
Suppose that $\Box \Diamond \varphi \to \varphi \in T$ and $\NS \models \pi'_T(\Box \Diamond \varphi)$. 
Then, $T \vdash \Diamond \varphi$, and $T \vdash \Box \Diamond \varphi$. 
Since $T \vdash \Box \Diamond \varphi \to \varphi$, we have $T \vdash \varphi$. 
Since $T$ is a subtheory of $\PA(\Triv)$, $T$ is an conservative extension of $\PA$ by Corollary \ref{Triv}. 
Then, $\PA \vdash \varphi$ because $\varphi$ is an $\mathcal{L}_A$-sentence. 
By the $\mathcal{L}_A$-soundness of $\PA$, we have $\NS \models \varphi$ and hence $\NS \models \pi_T'(\varphi)$. 
We have proved $\NS \models \pi'_U(\Box \Diamond \varphi \to \varphi)$. 
\end{proof}

In contrast to Corollary \ref{egLB-sound}, we have the following proposition which is a refinement of Proposition \ref{nDP}.1.

\begin{prop}
Let $T$ be a consistent r.e.~$\LB$-theory extending the theory $\PA(\KT) + \{\Box \Diamond \varphi \to \varphi \mid \varphi$ is a $\SB$-sentence$\}$. 
Then, $T$ does not have $(\B, \Sigma_1)$-$\DP$. 
\end{prop}
\begin{proof}
Let $\varphi$ be a $\Sigma_1$ sentence such that $T \nvdash \varphi$ and $T \nvdash \neg \varphi$. 
Then, $T \nvdash \Box \neg \varphi$ because $T \vdash \Box \neg \varphi \to \neg \varphi$. 
Since $T \vdash \Box \Diamond \varphi \to \Diamond \varphi$, we have $T \vdash \Diamond \Box \neg \varphi \lor \Diamond \varphi$. 
Then, $T \vdash \Diamond (\Box \neg \varphi \lor \varphi)$ and hence $T \vdash \Box \Diamond (\Box \neg \varphi \lor \varphi)$. 
Since $\Box \neg \varphi \lor \varphi$ is a $\SB$ sentence, we obtain $T \vdash \Box \neg \varphi \lor \varphi$ because $\Box \Diamond (\Box \neg \varphi \lor \varphi) \to (\Box \neg \varphi \lor \varphi)$ is an axiom of $T$. 
We have shown that $T \nvdash \Box \neg \varphi$, $T \nvdash \varphi$, and $T \vdash \Box \neg \varphi \lor \varphi$. 
This means that $T$ does not have $(\B, \Sigma_1)$-$\DP$. 
\end{proof}

\subsection{$\SB$-soundness and weak $\SB$-soundness}

We then export the notion of the $\Sigma_1$-soundness to modal arithmetic. 
This is easy to do since we have already introduced the class $\SB$ corresponding to $\Sigma_1$ in modal arithmetic. 
Here we further introduce another type of translation $\rho_T$, which is different from $\pi_T$.

\begin{definition}[$\rho$-translations]
Let $T$ be any r.e.~$\LB$-theory. 
We define a translation $\rho_T$ of $\LB$-formulas into $\mathcal{L}_A$-formulas inductively as follows: 
\begin{enumerate}
	\item If $\varphi$ is an $\mathcal{L}_A$-formula, then $\rho_T(\varphi) : \equiv \varphi$; 
	\item $\rho_T$ preserves logical connectives and quantifiers; 
	\item $\rho_T(\Box \varphi(\vec{x})) : \equiv \PR_T(\ulcorner \Box \varphi(\vec{\dot{x}}) \urcorner)$. 
\end{enumerate}
\end{definition}

With respect to $\SB$ sentences, there is the following relationship between the translations $\pi_T$ and $\rho_T$. 

\begin{prop}\label{piVSrho}
Let $T$ be any r.e.~$\LB$-theory.  
\begin{enumerate}
	\item For any $\SB$-sentence $\varphi$, $\NS \models \pi_T(\varphi) \to \rho_T(\varphi)$; 
	\item If $T$ is closed under the box elimination rule, then for any $\SB$-sentence $\varphi$, $\NS \models \rho_T(\varphi) \to \pi_T(\varphi)$.  
\end{enumerate}
\end{prop}
\begin{proof}
These statements are proved by induction on the construction of $\varphi$. 
We only prove the case of $\varphi \equiv \Box \psi$. 

1. If $\NS \models \pi_T(\Box \psi)$, then $\NS \models \PR_T(\gn{\psi})$. 
Then, $T \vdash \psi$. 
By the rule \textsc{Nec}, $T \vdash \Box \psi$. 
Then, $\NS \models \PR_T(\gn{\Box \psi})$, and hence $\NS \models \rho_T(\Box \psi)$. 

2. If $\NS \models \rho_T(\Box \psi)$, then $T \vdash \Box \psi$. 
By the box elimination rule, $T \vdash \psi$. 
Hence, $\NS \models \pi_T(\Box \psi)$. 
\end{proof}

We strengthen the usual $\Sigma_1$-completeness theorem of $\PA$ as follows. 

\begin{thm}[$\SB$-completeness theorem]\label{SBcompl2}
Let $T$ be any r.e.~extension of $\PA_\Box$. 
Then, for any $\SB$ sentence $\varphi$, if $\NS \models \rho_T(\varphi)$, then $T \vdash \varphi$. 
\end{thm}
\begin{proof}
We prove the theorem by induction on the construction of $\varphi$. 
\begin{itemize}
	\item If $\varphi$ is a $\Sigma_1$ sentence, then the statement immediately follows from the usual $\Sigma_1$-completeness of $\PA$ because $\rho_T(\varphi)$ is exactly $\varphi$. 
	\item If $\varphi$ is of the form $\Box \psi$, then $\NS \models \rho_T(\Box \psi)$ means $\NS \models \PR_T(\gn{\Box \psi})$, and hence $T \vdash \Box \psi$. 
	\item If $\varphi$ is one of the forms $\psi \land \sigma$, $\psi \lor \sigma$, $\exists x \psi$, and $\forall x < t \ \psi$, then the proof is straightforward by the induction hypothesis. 
\end{itemize}
\end{proof}

In the light of Proposition \ref{piVSrho} and Theorem \ref{SBcompl2}, we introduce the following two different types of the notion of $\SB$-soundness.

\begin{definition}
Let $T$ be any r.e. $\LB$-theory. 
\begin{itemize}
	\item $T$ is said to be \textit{$\SB$-sound} if for any $\SB$ sentence $\varphi$, if $T \vdash \varphi$, then $\NS \models \pi_T(\varphi)$; 
	\item $T$ is said to be \textit{weakly $\SB$-sound} if for any $\SB$ sentence $\varphi$, if $T \vdash \varphi$, then $\NS \models \rho_T(\varphi)$. 
\end{itemize}
\end{definition}

\begin{lem}\label{twoSBS}
For any r.e.~$\LB$-theory $T$, the following are equivalent: 
\begin{enumerate}
	\item $T$ is $\SB$-sound. 
	\item $T$ is weakly $\SB$-sound and $T$ is closed under the box elimination rule. 
\end{enumerate}
\end{lem}
\begin{proof}
By Proposition \ref{piVSrho}, it suffices to show that $\SB$-soundness implies the box elimination rule. 
Suppose that $T$ is $\SB$-sound. 
Let $\varphi$ be any $\LB$-sentence such that $T \vdash \Box \varphi$. 
By the $\SB$-soundness of $T$, $\NS \models \pi_T(\Box \varphi)$ and hence $\NS \models \PR_T(\gn{\varphi})$. 
We obtain $T \vdash \varphi$. 
\end{proof}

We are ready to prove an analogue of Guaspari's theorem.

\begin{thm}\label{weakSBsound}\leavevmode
Let $T$ be an r.e.~extension of $\PA_\Box$. 
\begin{enumerate}
	\item If $T$ contains $\PA(\K)$, $T \nvdash \Box \bot$, and $T$ has $\SB$-$\DP$, then $T$ is weakly $\SB$-sound. 
	\item 	If $T$ is weakly $\SB$-sound, then $T$ has $\SB$-$\EP$. 
\end{enumerate}
\end{thm}
\begin{proof}
1. We prove by induction on the construction of $\varphi$ that for any $\SB$ sentence $\varphi$, if $T \vdash \varphi$, then $\NS \models \rho_T(\varphi)$. 

\begin{itemize}
	\item If $\varphi$ is a $\Sigma_1$ sentence, then $\rho_T(\varphi) \equiv \varphi$. 
	Suppose $T \vdash \varphi$. 
	Since $\SB$-$\DP$ implies $(\B, \Sigma_1)$-$\DP$, $T$ is $\Sigma_1$-sound by Proposition \ref{DPtoS1sound}. 
	Therefore, $\NS \models \rho_T(\varphi)$. 

	\item If $\varphi$ is of the form $\Box \psi$, then $\rho_T(\varphi) \equiv \PR_T(\gn{\Box \psi})$. 
	Suppose $T \vdash \Box \psi$. 
	Then, obviously $\NS \models \rho_T(\varphi)$. 

	\item If $\varphi$ is of the form $\psi \land \sigma$ or $\forall x < t\, \psi$, then the proof is straightforward from the induction hypothesis. 

	\item If $\varphi$ is $\psi \lor \sigma$, then $\rho_T(\varphi) \equiv \rho_T(\psi) \lor \rho_T(\sigma)$. 
	Suppose $T \vdash \psi \lor \sigma$. 
	Then, by $\SB$-$\DP$, $T \vdash \psi$ or $T \vdash \sigma$. 
	By the induction hypothesis, $\NS \models \rho_T(\psi)$ or $\NS \models \rho_T(\sigma)$. 
	Hence, $\NS \models \rho_T(\varphi)$.

	\item If $\varphi$ is $\exists x \psi(x)$, then $\rho_T(\varphi) \equiv \exists x \rho_T(\psi(x))$. 
	Suppose $T \vdash \exists x \psi(x)$. 
	Since $T \nvdash \Box \bot$, by Theorem \ref{SBDP}, $T$ has $\SB$-$\EP$. 
	Then, there exists a natural number $n$ such that $T \vdash \psi(\overline{n})$. 
	By the induction hypothesis, $\NS \models \rho_T(\psi(\overline{n}))$. 
	Therefore, $\NS \models \rho_T(\varphi)$. 
\end{itemize}

2. Let $\varphi(x)$ be any $\SB$ formula having no free variables except $x$ such  that $T \vdash \exists x \varphi(x)$. 
By the weak $\SB$-soundness of $T$, $\NS \models \rho_T(\exists x \varphi(x))$. 
Then, for some natural number $n$, $\NS \models \rho_T(\varphi(\overline{n}))$. 
By Theorem \ref{SBcompl2}, $T \vdash \varphi(\overline{n})$. 
\end{proof}

\begin{cor}\label{SBsound}
Let $T$ be an r.e.~extension of $\PA_\Box$. 
\begin{enumerate}
	\item If $T$ contains $\PA(\KF)$, $T$ is consistent, and $T$ has $\MDP$, then $T$ is $\SB$-sound. 
	\item 	If $T$ is $\SB$-sound, then $T$ has $\MEP$. 
\end{enumerate}
\end{cor}
\begin{proof}
1. Since $T$ contains $\PA(\KF)$ and $T$ has $\MDP$, by Proposition \ref{MDPtoSBDP}, $T$ has $\SB$-$\DP$. 
Also, $T$ is closed under the box elimination rule by Proposition \ref{DPEPb}. 
Then, $T \nvdash \Box \bot$ by the consistency of $T$. 
By Theorem \ref{weakSBsound}.1, $T$ is weakly $\SB$-sound. 
By Lemma \ref{twoSBS}, $T$ is $\SB$-sound. 

2. By Lemma \ref{twoSBS}, $T$ is weakly $\SB$-sound and is closed under the box elimination rule. 
By Theorem \ref{weakSBsound}.2, $T$ has $\SB$-$\EP$. 
Therefore, by Proposition \ref{DPEPb}, $T$ has $\MEP$. 
\end{proof}

Since the $\LB$-soundness implies the $\SB$-soundness, we obtain the following corollary from Propositions \ref{ax-soundness1} and \ref{ax-soundness2}. 

\begin{cor}\label{MDP}\leavevmode
\begin{enumerate}
	\item Let $T$ be an $\LB$-theory obtained by adding some axioms of $\PA(\GL)$ into $\PA_\Box$, and let $U$ be any r.e.~extension of $T$. 
If $\NS \models \pi_U(\varphi)$ for all $\varphi \in U \setminus T$, then $U$ has $\MEP$. 
	\item Let $T$ be an $\LB$-theory obtained by adding some axioms of $\PA(\SF)$ into $\PA_\Box$, and let $U$ be any r.e.~extension of $T$. 
If $\NS \models \pi'_U(\varphi)$ for all $\varphi \in U \setminus T$, then $U$ has $\MEP$. 
\end{enumerate}
In particular, $\PA_\Box$, $\PA(\K)$, $\PA(\KT)$, $\PA(\KF)$, $\PA(\SF)$, and $\PA(\GL)$ have $\MEP$. 
\end{cor}

By Lemma \ref{twoSBS}, each $\SB$-sound theory is also weakly $\SB$-sound. 
Therefore, $\PA_\Box$, $\PA(\K)$, $\PA(\KT)$, $\PA(\KF)$, $\PA(\SF)$, and $\PA(\GL)$ also have $\SB$-$\EP$. 
Recall that $\PA(\Ver)$ also has $\SB$-$\EP$ (Corollary \ref{VerEP}). 

Here we give another sufficient condition for a theory to have $\SB$-$\EP$.
First, we prove an analogue of Proposition \ref{pi} with respect to $\rho$-translations. 

\begin{prop}\label{rho}
Let $T$ be any $\LB$-theory obtained by adding some axioms of the form $\forall \vec{x}(\Box \psi_0 \land \cdots \land \Box \psi_{k-1} \to \Box \psi_k)$ into $\PA_\Box$. 
Then, for any $\LB$-formula $\varphi$, if $T \vdash \varphi$, then for any r.e.~extension $U$ of $T$, $\PA \vdash \rho_U(\varphi)$. 
\end{prop}
\begin{proof}
Let $U$ be any r.e.~extension of $T$. 
As in the proof of Proposition \ref{alpha}, we prove by induction on the length of proofs of $\varphi$ in $T$ that for any $\LB$-formula $\varphi$, if $T \vdash \varphi$, then $\PA \vdash \rho_U(\varphi)$. 
We only give proofs of the following two cases. 

\begin{itemize}
	\item The case $\varphi \equiv \forall \vec{x} \bigl(\Box \psi_0(\vec{x}) \land \cdots \land \Box \psi_{k-1}(\vec{x}) \to \Box \psi_k(\vec{x}) \bigr)$: 
Since $U$ is an extension of $T$, we have $\PA \vdash \forall \vec{x} \, \PR_U(\gn{\Box \psi_0(\vec{\dot{x}}) \land \cdots \land \Box \psi_{k-1}(\vec{\dot{x}}) \to \Box \psi_k(\vec{\dot{x}})})$, and hence $\PA$ proves
\[
	\forall \vec{x} \bigl(\PR_U(\gn{\Box \psi_0(\vec{\dot{x}})}) \land \cdots \land \PR_U(\gn{\Box \psi_{k-1}(\vec{\dot{x}})}) \to \PR_U(\gn{\Box \psi_k(\vec{\dot{x}})}) \bigr).
\] 
This sentence is exactly $\rho_U(\varphi)$. 
	\item If $\varphi(\vec{x})$ is derived from $\psi(\vec{x})$ by \textsc{Nec}, then $\varphi(\vec{x}) \equiv \Box \psi(\vec{x})$. 
	Since $T \vdash \Box \psi(\vec{x})$, $U \vdash \Box \psi(\vec{x})$, and hence $\PA \vdash \PR_U(\gn{\Box \psi(\vec{\dot{x}})})$. 
	Thus, $\PA \vdash \rho_U(\varphi(\vec{x}))$. 
\end{itemize}
\end{proof}

\begin{cor}\label{SBEP}
Let $T$ be any $\LB$-theory obtained by adding some axioms of the form $\forall \vec{x}(\Box \psi_0 \land \cdots \land \Box \psi_{k-1} \to \Box \psi_k)$ into $\PA_\Box$, and let $U$ be any r.e.~extension of $T$. 
If $\NS \models \rho_U(\varphi)$ for all $\varphi \in U \setminus T$, then $U$ has $\SB$-$\EP$. 
\end{cor}
\begin{proof}
Suppose that $\NS \models \rho_U(\varphi)$ for all $\varphi \in U \setminus T$. 
As in the proof of Proposition \ref{ax-soundness1}, it follows from Proposition \ref{rho} that $U$ is weakly $\SB$-sound. 
Then, by Theorem \ref{weakSBsound}, $U$ has $\SB$-$\EP$.
\end{proof}

\subsection{Applications}

In this subsection, as applications of our results we have obtained so far, we show two non-implications between the properties. 
Corollary \ref{VerEP} shows that in general, $\SB$-$\DP$ does not imply $\MDP$. 
The first application shows that this is also true for theories that do not contain $\PA(\Ver)$.

\begin{prop}\label{nclosed}\leavevmode
\begin{enumerate}
	\item There exists an r.e.~theory $T$ such that $\PA(\SF) \vdash T \vdash \PA(\KF)$, $T \nvdash \Box \bot$, $T$ has $\SB$-$\EP$, and $T$ does not have $\MDP$; 
	\item There exists an r.e.~theory $T$ such that $\PA(\Ver) \vdash T \vdash \PA(\GL)$, $T \nvdash \Box \bot$, $T$ has $\SB$-$\EP$, and $T$ does not have $\MDP$. 
\end{enumerate}
\end{prop}
\begin{proof}
1. Let $T$ be the theory $\PA(\KF) + \{\Box \neg \Box \bot\}$. 
Since $T$ is a subtheory of $\PA(\Triv)$, we have $T \nvdash \Box \bot$. 
By Corollary \ref{SBEP}, $T$ has $\SB$-$\EP$. 
Suppose, towards a contradiction, that $T \vdash \neg \Box \bot$. 
Then, by the $\SB$-deduction theorem (Corollary \ref{SBDT}), $\PA(\KF) \vdash \Box \neg \Box \bot \to \neg \Box \bot$. 
Then, $\PA(\Ver) \vdash \Box \neg \Box \bot \to \neg \Box \bot$. 
Since $\PA(\Ver) \vdash \Box \neg \Box \bot \land \Box \bot$, this contradicts the consistency of $\PA(\Ver)$. 
Therefore, $T \nvdash \neg \Box \bot$. 
Since $T \vdash \Box \neg \Box \bot$, $T$ does not have $\MDP$. 

2. Let $T: = \PA(\GL) + \{\Box \Box \bot\}$. 
By Corollary \ref{SBEP}, $T$ has $\SB$-$\EP$. 
Suppose, towards a contradiction, that $T$ proves $\Box \bot$. 
By the $\SB$-deduction theorem, $\PA(\GL) \vdash \Box \Box \bot \to \Box \bot$. 
Then, $\PA(\GL) \vdash \Box(\Box \Box \bot \to \Box \bot)$ and hence $\PA(\GL) \vdash \Box \Box \bot$. 
By $\MDP$ of $\PA(\GL)$ (Corollary \ref{MDP} and Proposition \ref{EPtoDP}.1), $\PA(\GL) \vdash \bot$. 
This is a contradiction. 
Therefore, $T \nvdash \Box \bot$. 
Since $T \vdash \Box \Box \bot$, $T$ does not have $\MDP$. 
\end{proof}

Unlike the notion of the soundness of $\mathcal{L}_A$-theories, Proposition \ref{nclosed}.1 shows that the $\LB$-soundness is not preserved by taking a subtheory because $\PA(\SF)$ is $\LB$-sound but $T$ is not $\SB$-sound.

The second application shows that $\SB$-$\DC$ does not imply $\B$-$\DP$ in general.

\begin{prop}\label{ex2}
There exists a consistent r.e.~extension $T$ of $\PA(\KF)$ satisfying the following two conditions: 
\begin{enumerate}
	\item $T$ is $\SB$-$\DC$; 
	\item $T$ does not have $\B$-$\DP$. 
\end{enumerate}
\end{prop}
\begin{proof}
Let $\varphi$ be a G\"odel sentence of $\PA$. 
Let $T : = \PA(\KF) + \{\Box \varphi \lor \Box \neg \varphi\}$, $T_0 : = \PA(\KF) + \{\Box \varphi\}$, and $T_1 : = \PA(\KF) + \{\Box \neg \varphi\}$. 
By the $\SB$-deduction theorem, it is shown that for any $\LB$-formula $\psi$, 
\begin{equation}\label{eq3}
	T \vdash \psi\ \text{if and only if both}\ T_0 \vdash \psi\ \text{and}\ T_1 \vdash \psi. 
\end{equation}
Suppose, towards a contradiction, $T_0 \vdash \Box \bot$. 
By the $\SB$-deduction theorem, $\PA(\KF)$ proves $\Box \varphi \to \Box \bot$. 
Since this is also provable in $\PA(\Triv)$, by Proposition \ref{alpha}, we have $\PA \vdash \alpha(\Box \varphi) \to \alpha(\Box \bot)$. 
Then, $\PA \vdash \neg \varphi$, a contradiction. 
Similarly, we can prove $T_1 \nvdash \Box \bot$. 

1. Let $\psi$ be any $\SB$ sentence such that $T \vdash \psi \lor \PR_T(\gn{\psi})$. 
Then, for $i \in \{0, 1\}$, $T_i \vdash \psi \lor \PR_T(\gn{\psi})$ by (\ref{eq3}). 
By Corollary \ref{SBEP}, $T_i$ has $\SB$-$\EP$, and hence has $(\SB, \Sigma_1)$-$\DP$. 
By Corollary \ref{SBDPDC}, $T_i$ is $\SB$-$\DC$. 
Therefore, $T_i \vdash \psi$. 
By (\ref{eq3}), we obtain $T \vdash \psi$. 
Thus, $T$ is also $\SB$-$\DC$. 

2. If $T \vdash \Box \neg \varphi$ or $T \vdash \Box \varphi$, then $T_0 \vdash \Box \bot$ or $T_1 \vdash \Box \bot$ by (\ref{eq3}). 
This is a contradiction. 
Therefore, $T \nvdash \Box \neg \varphi$ and $T \nvdash \Box \varphi$. 
On the other hand, $T \vdash \Box \varphi \lor \Box \neg \varphi$. 
Thus, $T$ does not have $\B$-$\DP$. 
\end{proof}

\section{Problems}

In the present paper, several properties related to the modal disjunction property in modal arithmetic are introduced, and the relationships between them are studied.
However, some of the properties have not yet been separated in some particular situation. 
In this section, we list several unsolved problems for further study.

In \ref{Sec:B-DB}, we introduced $\B$-$\DP$ and $\DB$-$\DP$. 
For theories which are closed under the box elimination rule, these properties are equivalent. 
However, we have not yet been successful in clarifying whether they are equivalent or not in general.
We propose the following problem.

\begin{prob}\leavevmode
\begin{enumerate}
	\item Does there exist an $\LB$-theory which has $\DB$-$\DP$ but does not have $\B$-$\DP$?
	\item For each $n \geq 2$, does there exist an $\LB$-theory which has $\B^n$-$\DP$ but does not have $\B^{n+1}$-$\DP$?
\end{enumerate}
\end{prob}

For any $\Sigma_1$-unsound r.e.~extension $T$ of $\PA(\Ver)$, $T$ has $\B$-$\DP$ but does not have $\SB$-$\DC$ (see Propositions \ref{Verprov} and \ref{SBDPEP}). 
On the other hand, for consistent r.e.~extensions of $\PA(\SF)$, $\B$-$\DP$ implies $\SB$-$\DC$ by Propositions \ref{DPEPb} and \ref{MDPtoSBDP} and Corollary \ref{SBDPDC}. 
We have not yet been sure whether $\B$-$\DP$ yields $\SB$-$\DC$ in general when $T \nvdash \Box \bot$.

\begin{prob}
Does there exist an $\LB$-theory $T$ such that $T \nvdash \Box \bot$, $T$ has $\B$-$\DP$, and $T$ is not $\SB$-$\DC$?
\end{prob}

In the statement of Proposition \ref{MDPtoSBDP}, it is assumed that $T$ is an extension of $\PA(\KF)$. 
It is not settled yet whether $\PA(\KF)$ can be replaced by $\PA(\K)$ in the statement. 

\begin{prob}
In the statements of Proposition \ref{MDPtoSBDP} and Corollaries \ref{FSrev} and \ref{SBsound}.1, can $\PA(\KF)$ be replaced by $\PA(\K)$?
\end{prob}

Proposition \ref{nclosed}.1 shows that there exists an $\LB$-unsound subtheory $T$ of $\PA(\SF)$. 
Related to this fact, we propose the following problem. 

\begin{prob}
Does there exist an $\LB$-unsound r.e.~subtheory of $\PA(\GL)$?
\end{prob}

Proposition \ref{ex2} shows that $\SB$-$\DC$ does not imply $\B$-$\DP$. 
We are not successful in determining whether the theory $T$ in the proof of Proposition \ref{ex2} is closed in the box elimination rule.
We then propose the following problem. 

\begin{prob}
Does there exist a consistent r.e.~$\LB$-theory $T$ such that $T$ is closed under the box elimination rule, $T$ is $\SB$-$\DC$, and $T$ does not have $\MDP$?
\end{prob}

\begin{remark}
Notice that if we define $T$ to be the theory $\PA(\KF) + \{\Box \varphi \lor \Box \neg \varphi\}$ for a $\Pi_1$ G\"odel sentence $\varphi$ of $\PA(\KF)$, then $T$ is not closed under the box elimination rule. 
This is because $T \vdash \Box (\Box \varphi \lor \neg \varphi)$ and $T \nvdash \Box \varphi \lor \neg \varphi$. 
For, if $T \vdash \Box \varphi \lor \neg \varphi$, then $\PA(\KF) \vdash \Box \neg \varphi \to (\Box \varphi \lor \neg \varphi)$. 
By Proposition \ref{pi}, $\PA \vdash \PR_{\PA(\KF)}(\gn{\neg \varphi}) \to (\PR_{\PA(\KF)}(\gn{\varphi}) \lor \neg \varphi)$. 
Since $\PR_{\PA(\KF)}(\gn{\varphi})$ is $\PA(\KF)$-equivalent to $\neg \varphi$, we have $\PA(\KF) \vdash \PR_{\PA(\KF)}(\gn{\neg \varphi}) \to \neg \varphi$. 
By L\"ob's theorem, $\PA(\KF) \vdash \neg \varphi$. 
This contradicts the $\Sigma_1$-soundness of $\PA(\KF)$. 
\end{remark}

\bibliographystyle{plain}
\bibliography{ref}

\end{document}